\documentclass[11pt,oneside]{amsart}

\usepackage{amssymb}

\usepackage{amsmath,amscd}
\usepackage{hyperref}
\usepackage{fontenc}
\usepackage{textcomp}
\usepackage[margin=1in]{geometry}
\usepackage{inputenc}
 \theoremstyle{plain}
      \newtheorem{theorem}{Theorem}[section]
      \newtheorem{lemma}[theorem]{Lemma}

      \newtheorem{coro}[theorem]{Corollary}

      \newtheorem{prop}[theorem]{Proposition}
      \newtheorem{defi}[theorem]{Definition}

      \numberwithin{equation}{section}

      \theoremstyle{plain}
      \newtheorem{remark}[theorem]{Remark}
      
\begin{document}

%



  \author{Adrien Boyer } 
   \thanks{ aadrien.boyer@gmail.com, Technion, Haifa, Israel.\\
   This work is supported by ERC Grant 306706. }

   \title[Square root of Poisson kernel 
   in $\delta$-hyperbolic spaces]{A theorem \`a la Fatou for the square root of Poisson kernel 
   in $\delta$-hyperbolic spaces and decay of matrix coefficients of boundary representations}

   \begin{abstract}
 In this note, we prove a theorem \`a la Fatou for the square root of Poisson Kernel in the context of quasi-convex cocompact discrete groups of isometries of $\delta$-hyperbolic spaces. As a corollary we show that some matrix coefficients of boundary representations cannot satisfy \emph{the weak inequality of Harish-Chandra}. Nevertheless, such matrix coefficients satisfy an inequality which can be viewed as a particular case of the inequality coming from \emph{property RD} for boundary representations. The inequality established in this paper is based on a uniform bound which appears in the proof of the irreducibility of boundary representations. Moreover this uniform bound can be used to prove that the Harish-Chandra's Schwartz space associated with some discrete groups of isometries of $\delta$-hyperbolic spaces carries a natural structure of a convolution algebra.
  Then in the context of CAT(-1) spaces we show how our elementary techniques enable us to apply an equidisitribution theorem of T. Roblin to obtain informations about the decay of matrix coefficient of boundary representations associated with continuous functions.  
    \end{abstract}

   \subjclass[2010]{Primary 31A20, 37F35, 20F67; Secondary 42A20, 43A90, 37A35, 22AD10, 22AD15}

   \keywords{Theorem \`a la Fatou, quasi-conformal densities, square root of the Poisson kernel, boundary representations, Harish-Chandra functions, decay of matrix coefficients, Harish-Chandra's Schwartz space of discrete groups, property RD}


   \dedicatory{}



   \maketitle

\section{Introduction}
On the unit disc $\mathbb{D}$, the Poisson transform of a function $f$ in $L^{1}(\partial \mathbb{D},\frac{{\rm d}\theta}{2\pi})$ is defined by the formula $$Pf(z)=\int_{\partial \mathbb{D}} f(\theta)P(z,\theta) \frac{ {\rm d}\theta}{2\pi},$$ where $P(z,\theta)$ is the standard Poisson kernel and $\frac{{\rm d}\theta}{2\pi}$ is the normalized Lebesgue measure on the unit circle. It is well known that the harmonic function $Pf$   converges to $f$ for almost
all $\theta \in \partial \mathbb{D}$ as $z\rightarrow {\rm e}^{i \theta}$ and $z$ stays in the nontangential cone which one can define by
the inequality $| {\rm arg }(z) -\theta| < C (1 - |z|)$ where $C$ is a positive constant.

However, the situation is different when we consider \emph{the (normalized) square root of the Poisson kernel}. Let $f$ in $L^{1}(\partial \mathbb{D},\frac{{\rm d}\theta}{2\pi})$ and define the function $\mathcal{P}_{0}f$ by the formula $$\mathcal{P}_{0}f(z)=\int_{\partial \mathbb{D}} f(\theta)\frac{P^{\frac{1}{2}}(z,\theta)}{\int_{\partial X}P^{ \frac{1}{2} }(z,\theta) \frac{ {\rm d}\theta}{2\pi}}  \frac{ {\rm d}\theta}{2\pi} \cdot$$ 
\\
P. Sj\"ogren proved in a short note \cite{Sjo2} that the function $\mathcal{P}_{0}f$  converges to $f$ for almost
all $\theta \in \partial \mathbb{D}$ as $z\rightarrow {\rm e}^{i \theta}$ and $z$ stays in the weak nontangential domain which one can define by
the inequality $$| {\rm arg }(z) -\theta| < C (1 - |z|) \bigg(  \log \frac{1}{1-|z|}\bigg)$$ for some positive constant $C>0$.

These convergence results, also called \emph{theorem \`a la Fatou}, have been studied in \cite{Sjo}, \cite{Ron1} and generalized to the context of symmetric spaces in \cite{M}, \cite{Ron2} and in \cite{Sjo2}. Generalizations of theorems \emph{\`a la Fatou} have already been done in the context of CAT(-1) spaces or negative curvature as we can see in \cite{Ro} or in \cite{Alv}. We generalize this \emph{theorem \`a la Fatou for the normalized square root of the  Poisson kernel} to the context of Patterson-Sullivan density associated with quasi-convex cocompact discrete group of isometries of $\delta$-hyperbolic spaces. More precisely, this work is based on the framework of Patterson-Sullivan measures developed by M. Coornaert in \cite{Coo} for general $\delta$-hyperbolic spaces. Besides, the Poisson kernel being the conformal factor of an isometry acting on the boundary, we make connections with boundary representations and this theorem \`a la Fatou. We obtain information about decay of matrix coefficients associated with boundary representations and we will also deal with analogs of \emph{Harish-Chandra's function}, \emph{Harish-Chandra's weak inequality} and \emph{Harish-Chandra's Schwartz algebras} in the context of discrete groups of isometries of $\delta$-hyperbolic spaces.\\

\subsection{Notation and statement of the results}
 
Let $\Gamma$ be a nonelementary discrete group of isometries of a metric space $(X,d)$ which is $\delta$-hyperbolic. We denote by $\partial X$ its Gromov boundary and let $\overline{X}$ be the topological space $X \cup \partial X$ endowed with its usual topology that makes $\overline{X}$ compact. 

It is well known (see \cite{Gro} and \cite[Chapter III.H, Proposition 3.21]{BH}, and see Section \ref{prelim} in this paper) that the topology on $\partial X$ is metrizable: once a base point $x$ in $X$ has been fixed, there exists a parametrized family $(d^{\epsilon}_{x})_{\epsilon>0}$  of metrics each of which are called a \emph{visual metric} of visual parameter $\epsilon$.

Recall the critical exponent $\alpha(\Gamma)$ of $\Gamma$: $$\alpha(\Gamma):=\inf \left\{s\in \mathbb{R}^{*}_{+} | \sum_{\gamma \in \Gamma} {\rm e}^{-sd(\gamma x,x)} <\infty \right\}.$$Notice that the definition of $\alpha(\Gamma)$ does not depend on $x$. We assume from here on now that $\alpha(\Gamma)<\infty$.

The limit set of $\Gamma$ denoted by $\Lambda_{\Gamma}$ is the set of all accumulation points in $\partial X$ of an orbit. Namely $\Lambda_{\Gamma}:=\overline{\Gamma x}\cap \partial X$, with the closure in $\overline{X}$. Notice that the limit set does not depend on the choice of $x\in X$. Following the notations in \cite{CM}, define the geodesic hull $GH(\Lambda_{\Gamma})$ as the union of all geodesics in $X$ with both endpoints in $\Lambda_{\Gamma}$. The convex hull of $\Lambda_{\Gamma}$ denoted by $CH(\Lambda_{\Gamma})$, is the smallest subset of $X$ containing $GH(\Lambda_{\Gamma})$ with the property that every geodesic segment between any pair of points $x$ and $y$ in CH$(\Lambda_{\Gamma})$ also lies in $CH(\Lambda_{\Gamma})$. We say that $\Gamma$ is \emph{quasi-convex cocompact} (respectively \emph{convex cocompact}) if it acts cocompactly on  $GH(\Lambda_{\Gamma})$ (respectively on $CH(\Lambda_{\Gamma})$). 
 
The foundations of Patterson-Sullivan measures theory are in the important papers \cite{Pa} and \cite{Su}. See \cite{Bou},\cite{BMo} and \cite{Ro} for more general results in the context of CAT(-1) spaces. These measures are also called \textit{conformal densities}.
However, in this paper we will consider a more general measure class, called \emph{quasi-conformal densities}. \\
We denote by $M(Z)$ the Banach space of Radon measures on a locally compact space $Z$, which is identified as the dual space of compactly supported functions denoted by $C_{c}(Z)^*$, endowed with the norm $\|\mu\|=\sup\lbrace |\int_{Z}f d\mu|, \|f\|_{\infty}\leq 1,f \in C_{c}(Z) \rbrace$ where $\|f\|_{\infty}=\sup_{ z\in Z} |f(z)|$.
Recall that $\gamma_{*}\mu$ means $\gamma_{*}\mu(B)=\mu(\gamma^{-1}B)$ where $\gamma$ is in $\Gamma$ and $B$ is a borel subset of $Z$. 
\begin{defi}\label{quasiconform}
We say that $\mu$ is a \emph{$\Gamma$-invariant quasi-conformal density  of dimension $\alpha \geq 0$}, if $\mu$ is a map which satisfies the following conditions:
\begin{enumerate}
	\item $\mu$ is a map from $x\in X \mapsto \mu_{x} \in M(\overline{X})$, i.e. $\mu_{x}$ is a positive finite measure (density).
	\item There exists $C_{q}>0$ such that for all $x$ and $y$ in $X$, $\mu_{x}$ and $\mu_{y}$ are equivalent and they satisfy almost everywhere: $$  C_{q}^{-1}{\rm e}^{\alpha \beta_{v}(x,y)}\leq \frac{d\mu_{y}}{d\mu_{x}}(v)\leq C_{q}{\rm e}^{\alpha \beta_{v}(x,y)}$$ (quasi-conformal of dimension $\alpha$).
	\item For all $\gamma \in \Gamma$ and for all $x \in X$ we have $\gamma_{*}\mu_{x}=\mu_{\gamma x}$  (invariant),
\end{enumerate}
where $\beta_{v}(x,y)$ denotes a Busemann function associated with $x$ to $y$ relative to $v$ (see Section \ref{prelim}).

Moreover if $$(x,y,v) \in X\times X \times \partial X \mapsto  \frac{d\mu_{y}}{d\mu_{x}}(v)$$ is continuous, we say that $\mu$ is a \emph{$\Gamma$-invariant continuous quasi-conformal density  of dimension $\alpha \geq 0$}.

If in $(2)$ we have $C_{q}=1$ and that the equality holds everywhere, we say that $\mu$ is \emph{$\Gamma$-invariant conformal density  of dimension $\alpha \geq 0$}.

\end{defi}

In \cite[Th\'eor\`eme 8.3]{Coo}, M. Coornaert proves the existence of such $\Gamma$-invariant quasi-conformal density of dimension $\alpha(\Gamma)$
 when $X$ is a proper geodesic $\delta$-hyperbolic space and when the group $\Gamma$ is quasi-convex cocompact (as it has been mentioned, after a slight modification, by C. Connell and R. Muchnik \cite[Proposition 1.23 and Corollary 1.24]{CM}).

 	If $X$ is a \mbox{CAT(-1)} space and if $\Gamma$ is a discrete group of isometries of $X$, then there exists a $\Gamma$-invariant conformal density of dimension $\alpha(\Gamma)$ whose support is $\Lambda_{\Gamma}$. A proof of this theorem can be found in \cite{Pa} and \cite{Su} for the case of hyperbolic spaces and see \cite{BM} and \cite{Bou} for the case of \mbox{CAT(-1)} spaces.

\subsection*{The Poisson kernel}
 Given $\mu$ a $\Gamma$-invariant quasi-conformal density of dimension $\alpha $ we define the Poisson kernel in the context of $\delta$-hyperbolic spaces associated with the measure class $\mu$ as:
\begin{equation}\label{Poissonkernel}
   P: (x,y,v)\in X\times X \times \partial{X} \mapsto P(x,y,v):=\bigg(\frac{d\mu_{y}}{d\mu_{x}} (v) \bigg)^{1 / \alpha}\in \mathbb{R}^{+}.
\end{equation}

 \begin{remark}
 In the case of CAT(-1) spaces, we have $ P(x,y,v)={\rm e}^{ \beta_v(x,y)}$. Thus in the case of hyperbolic spaces, we find the standard Poisson kernel. Indeed this definition of the Poisson kernel is suggested by the definition of the Poisson kernel as the conformal factor in \cite[(2.7.1) p. 92]{Bou} or as in \cite{Pa}.
 \end{remark}

We fix an origin $x\in X$ and imitating the notations of Sj\"ogren (\cite{Sjo}) we set for $\lambda\in\mathbb R$
and $f\in L^1(\partial X,\mu_x)$
\begin{equation*}\label{sjogrennot}
   P_{\lambda,\mu_{x}}f  (y):=\int_{\partial X}  P(x,y,v)^{\alpha (\lambda+1/2)}f(v)d\mu_x(v).
\end{equation*}
\begin{defi}
Define for $\lambda \in \mathbb{R}$ the normalized Poisson kernel as
 $$  \big( \mathcal{P}_{\lambda,\mu_{x}}f \big):y \in X \mapsto \frac{P_{\lambda,\mu_{x}}f(y)}{P_{\lambda,\mu_{x}}\textbf{1}_{\partial X}(y)}\in \mathbb{C}.$$
\end{defi}
where $\textbf{1}_{\partial X}$ is the characteristic function of $\partial X$ and $y$ is in $X$.

In this article, for $\lambda=0$ we will write  $\mathcal{P}_{0}f$ rather than $\mathcal{P}_{0,\mu_{x}}f$ and we will define and study the weak radial convergence called \emph{weak nontangential convergence} of $\mathcal{P}_{0}f$.   
\subsection*{Upper Gromov product bounded by above and weak nontangential convergence}

To define \emph{weak nontangential convergence} on $\overline{X}$ we use the notion of \emph{upper Gromov product bounded by above}, notion which has already been emphasized in \cite{CM}. 
Following \cite[Definition 1.18]{CM}, we say that a $\delta$-hyperbolic space $X$ is \emph{upper Gromov product bounded by above} if there exists a point $x\in X$ and a positive constant $c$ such that for all $y\in X$:
\begin{equation}
\sup_{v\in \partial X} (y,v)_{x} \geq d(x,y)-c,
\end{equation}
where $(y,z)_{x}$ denotes the Gromov product of $y,v\in \overline{X}$ with respect to $x\in X$.\\
Notice that if the above inequality holds for some point $x$, the same inequality holds for all $x$ in $X$ (with another constant $c$).
 
This property on $X$ enables us to choose a point 
\begin{equation}
w_{x}^{y} \in\partial X,
\end{equation}
such that 
\begin{equation}\label{gromovbounded}
  (y,w_{x}^{y})_{x}  \geq d(x,y)-c.
\end{equation}
 The point $w_{x}^{y} $ plays the role of the unique ending point of the oriented geodesic passing through $x$ and $y$ in the context of connected negatively curved manifolds. 
 
\begin{defi}\label{nonrad} Let $\mu$ be a $\Gamma$-invariant quasi-conformal density of dimension $\alpha$. Fix an origin $x\in X$ and
let $v\in \Lambda_{\Gamma}$. We define \emph{a weak nontangential approach domain} of $v$ as $$\Omega_{C}(v):=\big\lbrace  y\in GH(\Lambda_{\Gamma})  \mbox{ $|$ } d_{x}^{\epsilon}(w_{x}^{y},v) \leq C d(x,y)^{\frac{\epsilon}{\alpha}} {\rm e}^ {  - \epsilon d(x,y) } \big\rbrace .$$ 

We say that a function $F:X \rightarrow \mathbb{C}$ has a \emph{weak nontangential limit $L$} at $v\in \partial X $ if for all $C>0$ we have $F(y)\rightarrow L$ as $y\rightarrow v$ with respect to the topology on $\overline{X}$  and if $y \in \Omega_{C}(v)$.
\end{defi}

We obtain a theorem \`a la Fatou for the square root of the Poisson kernel in the context of discrete groups of isometries of $\delta$-hyperbolic spaces:

\begin{theorem}\label{Fatou}
Let $X$ be a proper geodesic $\delta$-hyperbolic space which is upper Gromov product bounded by above and assume that the Gromov boundary $\partial X$ is endowed with a visual metric $(d_{x}^{\epsilon})_{x\in X}$ of visual parameter $ 0<\epsilon\leq 1$.
Let $\Gamma$ be a discrete quasi-convex cocompact group of isometries of $X$ and let $\mu$ be a $\Gamma$-invariant continuous quasi-conformal density of dimension $\alpha(\Gamma)$. \\
For each $x$ in $X$ and for all $f\in L^{1}(\partial X,\mu_{x})$ the function 
 $\mathcal{P}_{0}f$ has \emph{weak non tangential limit} $f$ almost everywhere with respect to $\mu_{x}$.  
\end{theorem}

We obtain immediately the following corollary:
\begin{coro}\label{Fatou1}
Let $X$ be a proper CAT(-1) space and let $\Gamma$ be a convex cocompact discrete group of isometries of $X$. Let $\mu=(\mu_{x})_{x\in X}$ be a $\Gamma$-invariant conformal density of dimension $\alpha(\Gamma)$. \\
For each $x$ in $X$ and for all $f\in L^{1}(\partial X,\mu_{x})$ we have 
 $\mathcal{P}_{0}f$ has \emph{weak non tangential limit} $f$ almost everywhere with respect to $\mu_{x}$.  
\end{coro}

The connection with boundary representations will be almost immediate.

\subsection*{Boundary representations }
	
  A $\Gamma$-invariant quasi-conformal density $\mu$ of dimension $\alpha$ gives rise to unitary representations $(\pi_{x})_{x\in X}$ defined for all $x\in X$ as:
 
	\[
		\pi_x:\Gamma\to \mathcal{U}\big(L^2(\partial X,\mu_x) \big)
	\]
	\begin{equation}\label{boundaryrep}
	   (\pi_{x}(\gamma)\xi)(v)=P(x,\gamma x,v)^{\alpha / 2}\xi(\gamma^{-1}v),
	   	\end{equation}
		
where $\xi \in L^{2}(\partial X,\mu_{x})$ and $v\in \partial X$.

\begin{remark}
In CAT(-1) spaces these representations $(\pi_{x})_{x\in X}$ are unitarily equivalent. More generally, the boundary representations depend only on the measure class from which they are built.
\end{remark} 

We will focus our attention on some $x\in X$ \[
		\pi_x:\Gamma\to \mathcal{U}\big(L^2(\partial X,\mu_x) \big)
	\] where $x$ is an origin in $X$.
	
\subsection*{Harish-Chandra's function and growth}

One has to define a function on $\Gamma$ which is going to play a central role in this paper. The matrix coefficient 
\begin{equation}\label{HCHfunction}
\phi_{x}: \Gamma \rightarrow \langle \pi_{x}(\gamma)\textbf{1}_{\partial X}, \textbf{1}_{\partial X}\rangle \in \mathbb{R}^{+},
\end{equation} where $\textbf{1}_{\partial X}$ is the characteristic function of $\partial X$, is called the \emph{Harish-Chandra function}. Observe that the Harish-Chandra functions satisfies for all $\gamma$ in $\Gamma$ 
\begin{equation}\label{inverse}
\phi_{x}(\gamma)=\phi_{x}(\gamma^{-1}). 
\end{equation}

 We call it the Harish-Chandra function because its definition in the context of reductive Lie groups goes back to Harish-Chandra (see \cite[Chapter 3, \S 3.6.1]{W}).

 Following \cite[Definition 6.1.17]{GV}, we say that a function $f:\Gamma \rightarrow \mathbb{C}$ satisfies the \emph{weak inequality of Harish-Chandra of degree $t$} if for some $x$ in $X$, there exist $C>0$ and $t\geq 0$ such that
\begin{equation}
|f(\gamma)|\leq C\big(1+|\gamma|_{x}\big) ^{t}\phi_{x}(\gamma), 
\end{equation}
where $|\cdot|_{x}$ is a length function on $\Gamma$ defined as $|\gamma|_{x}:=d(x,\gamma x)$.

Indeed, it is not hard to see that if the weak inequality holds for some $x$, then it holds for all $x $ in $X$ with a constant $C$ depending on $x$ but with the same $t$ (see Section \ref{prelim}).

As a corollary of Theorem \ref{Fatou}, we have:

\begin{coro}\label{Fatou1}
 Let $\Gamma$ be a convex cocompact discrete group of isometries of a CAT(-1) space $X$ and let $\mu$ be a $\Gamma$-invariant conformal density of dimension $\alpha(\Gamma) $. For each $x$ in $X$, consider the boundary representations $\pi_{x}:\Gamma \rightarrow \mathcal{U}(L^{2}(\partial X,\mu_{x}))$.

For each $x$ in $X$ and for all $\xi$ in $L^{2}(\partial X,\mu_{x})$ such that $\xi$ is not in $L^{\infty}(\partial X)$, the matrix coefficient $ \langle \pi_{x}(\cdot)\textbf{1}_{\partial X},\xi \rangle$ does not satisfy \emph{the weak inequality of Harish-Chandra of degree 0}.
\end{coro}

Fix a point $x$ in $X$ and denote by $\varphi_{x}$ the function  
\begin{equation}\label{HCHcontinue}
\varphi_{x}:y\in X \mapsto P_{0,\mu_{x}}\textbf{1}_{\partial X} (y)=\int_{\partial X}P(x,y,v)^{\frac{\alpha}{2}}d\mu_{x}(v).
\end{equation}
Observe that the Harish-Chandra function $\phi_{x}$ is the restriction of $\varphi_{x}$ to the orbit $\Gamma x$.\\
Let $\mathcal{Y}$ be a subset of $X$. We say that $\varphi_{x}$ satisfies the \emph{Harish-Chandra estimates on $\mathcal{Y}$} if there exist two polynomials $Q_{1}$ and $Q_{2}$ of \emph{degree one with positive coefficients} such that for all $y\in \mathcal{Y}$ we have
\begin{equation}\label{HCHestim}
Q_{1}\big( d(x,y)\big) {\rm e}^{- \frac{\alpha}{2}d(x,y) }\leq \varphi_{x}(y) \leq Q_{2}\big( d(x,y)\big) {\rm e}^{-\frac{\alpha}{2}d(x,y) }.
\end{equation}

Define the class \begin{equation}\label{Class}
\mathcal{C}
 \mbox{ of discrete groups  as the class made of:}
  \end{equation}
$\bullet$ convex cocompact groups of isometries of a CAT(-1) space,\\
$\bullet$ nonuniform lattices in  noncompact semisimple Lie group of rank one acting by isometries on their rank one symmetric spaces of noncompact type where $d$ is a left invariant Riemanninan metric, \\
$\bullet$ hyperbolic groups $\Gamma$ acting by isometries on $(\Gamma,d_{S})$,
where $d_{S}$ is the word metric associated with a finite symmetric generating set $S$.

 \begin{remark}
The common point of these groups is that there exists a $\Gamma$-invariant quasi-conformal density of dimension $\alpha(\Gamma)$ denoted by $\mu$ such that the function $\varphi_{x}$ associated with $\mu_{x}$ satisfies the Harish-Chandra estimates on GH$(\Lambda_{\Gamma}) \backslash B_{X}(x,R)$ where $B_{X}(x,R)$ denotes the ball of $X$ centered at $x$ of radius $R$. 
\end{remark}
 We say that a function $f$ defined on a group $\Gamma$ endowed with a length function $l$ has polynomial growth if there exists a polynomial $Q$ such that for all integers $n$ 
\begin{equation*}
\sum_{l( \gamma )\leq n  } f(\gamma)\leq Q(n).
\end{equation*}

Notice that if a function $f$ on $\Gamma \in \mathcal{C}$ satisfies the weak inequality of Harish-Chandra of degree $t$ with respect to some Harish-Chandra function $\phi_{x}$ (associated with the $\Gamma$ invariant quasi-conformal density of dimension $\alpha(\Gamma)$) then the function $|f|^{2}$ has polynomial growth with respect to the length function $|\cdot |_{x}$. In particular the square of Harish-Chandra's function has polynomial growth.

Although Corollary \ref{Fatou1} asserts that some coefficients cannot decay too rapidly to $0$, we are able to show that these coefficients have only polynomial growth.
Let $\Gamma_{n}(x)$ be the ball of $\Gamma$ of radius $n$ centered at $x$ defined as $$\lbrace \gamma \in \Gamma \mbox{ such that } |\gamma|_{x} \leq n\rbrace.$$

\begin{prop}\label{propFatou}
Let $\Gamma$ be in $\mathcal{C}$ and consider their boundary representations associated with the $\Gamma$-invariant quasi-conformal density of dimension $\alpha(\Gamma)$.
Then for each $x$ in $X$ there exists a polynomial $Q$ (of degree $3$) such that for all unit vectors $\xi$ in $L^{2}(\partial X,\mu_{x})$ we have for all integers $n$  $$\sum_{\gamma \in \Gamma_{n}(x)} |\langle  \pi_{x}(\gamma)\textbf{1}_{\partial X},\xi\rangle |^{2} \leq Q(n).$$
\end{prop}

Besides, a direct application of an equidistribution theorem of T. Roblin gives another information about the sum of matrix coefficients associated with continuous functions (for a less general class of groups). We denote by $\|\cdot \|_{2}$ the $L^{2}$ norm of $L^{2}(\partial X, \mu_{x})$. Define the annulus $C_{n,\rho}(x)$ for $\rho>0$ and $n\geq \rho$ as 
 \begin{equation}
 \lbrace \gamma \in \Gamma \mbox{ such that } n-\rho \leq  |\gamma|_{x} \leq n+\rho \rbrace.
\end{equation}

We refer to \cite[Chapitre 1 Pr\'eliminaires, 1C. Flot g\'eod\'esique]{Ro} and \cite[Preliminaries 2.2]{B2} for details on the Bowen-Margulis-Sullivan measure and on the mixing property of the geodesic flow.  We obtain:

\begin{prop}\label{proprob}
Let $\Gamma$ be a discrete group of isometries belonging to one of the following classes:\\
$\bullet$ convex cocompact discrete groups of isometries of a CAT(-1) space such that the geodesic flow is mixing with respect to the Bowen-Margulis-Sullivan measure or
\\
$\bullet$ nonuniform lattices in noncompact rank one semisimple Lie groups acting by isometries on their symmetric spaces of noncompact type. \\

For each $x\in X$ there exist $\rho>0$ and a polynomial $Q$ (of degree $3$) such that for all continuous functions $f$ and $g$ in $C(\partial X)$ we have $$\limsup _{n \rightarrow +\infty }\frac{1}{Q(n)}\sum_{\gamma \in C_{n,\rho}(x)} |\langle  \pi_{x}(\gamma)f,g\rangle |^{2} \leq \|f \|^{2}_{2}\|g \|^{2}_{2}.$$
\end{prop}

\begin{remark}\label{RDrem}
Proposition \ref{propFatou} and Proposition \ref{proprob} are intimately related to the so-called \emph{property RD} (rapid decay) introduced by U. Haagerup in \cite{H} and studied as such by P. Jolissaint in \cite{Jol}. The inequality of Proposition \ref{propFatou} can be seen as a particular case of the inequality coming from property RD, with one of the two functions being trivial. In Proposition \ref{proprob} if instead the $\limsup$ we had the $\sup$, then we would have obtained property RD.
\end{remark}

\subsection*{Harish-Chandra's Schwartz algebra of discrete groups}
Let $\Gamma$ be a discrete group endowed with a length function denoted by $|\cdot |$ and let 
$\phi:\Gamma \rightarrow \mathbb{R}^{+}$ be a positive function on $\Gamma$. Assume that $\phi$ satisfies:
\begin{enumerate}\label{hyp hch0}
\item $\phi(e)=1,$
\item $\phi(\gamma^{-1})=\phi(\gamma)$, $\forall \gamma \in \Gamma.$
\item There exists $t>0 $ such that  $$\sum_{\gamma \in\Gamma} \frac{\phi^{2}(\gamma)}{(1+|\gamma|)^{t}}<\infty.$$
\end{enumerate}

We denote by $C_{c}(\Gamma)$ the space of finitely supported functions. 
\begin{defi}
 Let $t>0$ and define Harish-Chandra's Schwartz space $\mathcal{S}_{t}(\Gamma)$ associated with $(\Gamma, | \cdot |, \phi)$  as the Banach space completion of $C_{c}(\Gamma)$ with respect to the norm $$\|f\|_{\mathcal{S}_{t}(\Gamma)}=\sup_{\gamma \in \Gamma}\frac{|f(\gamma)|}{\phi(\gamma)}(1+|\gamma|)^{t}.$$

\end{defi}
Observe that if a function $f:\Gamma \rightarrow \mathbb{C}$ satisfies the weak inequality of degree $t\geq 0$ (with respect to $\phi$), then $f\in S_{t}(\Gamma)$.

Notice that assumption (3) guarantees that $S_{t}(\Gamma)\subset \ell ^{2}(\Gamma)$.

We can describe $\mathcal{S}_{t}(\Gamma)$ as $$\mathcal{S}_{t}(\Gamma)=\left\lbrace \sum_{\gamma \in \Gamma} c_{\gamma} \gamma \mbox{ such that }  c_{\gamma}=O{ \bigg(\frac{\phi(\gamma)}{(1+|\gamma |)^{t}} \bigg)}  \right\rbrace.$$ 

It turns out that Harish-Chandra's Schwartz space carries a natural structure of a convolution algebra and can be represented on $\ell^{2}(\Gamma)$.
We denote by $C^{*}_{r}(\Gamma)$ the reduced $C^{*}$-algebra of $\Gamma$ that is the $C^{*}$-algebra associated with the left regular representation of $\Gamma$ and by $W^{*}(\Gamma)$ the von Neumann algebra of $\Gamma$ that is the von Neumann algebra associated with the left regular representation of $\Gamma$. 
\begin{prop}\label{LG}
Let $\Gamma$ be in $\mathcal{C}$ and consider a $\Gamma$-invariant quasi-conformal density $(\mu_{x})_{x\in X}$ of dimension $\alpha(\Gamma)$. Then for each $x\in X$, there exists $t>0$ such that Harish-Chandra's Schwartz space $\mathcal{S}_{t}(\Gamma)$ associated with $(\Gamma,|\cdot |_{x},\phi_{x})$ is a $*$-Banach convolution algebra which is dense in $C^{*}_{r}(\Gamma)$ and in $W^{*}(\Gamma)$. 
\end{prop}

\subsection{Structure of the paper} 
 
 In Section \ref{prelim} we remind the reader of some standard facts about $\delta$-hyperbolic spaces, Ahlfors regularity of metric measure space, Harish-Chandra's function and maximal inequality of type $(1,1)$. \\In Section \ref{Fatouu} we prove our main result: Theorem \ref{Fatou}.  \\In Section \ref{applic} we deduce immediately Corollary \ref{Fatou1} from Theorem \ref{Fatou}. We recall a uniform boundedness satisfied by a certain sequence of functions which has already been used in the proof of the irreducibility of boundary representations (in \cite{BM}, \cite{LG} and \cite{B2}). We use this uniform boundedness to prove Proposition \ref{propFatou}. We recall a theorem of equidistribution of T. Roblin and combining this theorem with techniques involved in the proof of Theorem \ref{Fatou} we prove Proposition \ref{proprob}.\\ In Section \ref{HCS} we use the uniform boundedness of the previous section to prove Proposition \ref{LG}. It turns ou that the inequality of Proposition \ref{poissontransform'} is fundamental to prove that the Harish-Chandra's Schwartz space of discrete groups carries a natural structure of convolution algebra. This inequality plays analogous roles of the conditions $(HC3c)$ and $(HC3d)$ in \cite[Chapitre 4, p. 82]{L} satisfied by spherical functions in the context of semisimple Lie groups.
 \subsection*{Acknowledgements}
 I would like to thank Nigel Higson who suggested that I investigate the structure of convolution algebra of Harish-Chandra's Schwartz space of discrete groups. I would like also to thank Uri Bader and Amos Nevo for valuable discussions. And I am grateful to Felix Pogorzelski and Dustin Mayeda for their remarks and comments on this note.

\section{Preliminaries}\label{prelim}
\subsection{$\delta$-Hyperbolic spaces}
We follow \cite[Chapter III.H, p. 431]{BH} and \cite{G}, after \cite{Gro}.
 
Recall that the Gromov product of two points $y,z\in X$ relative to $x\in X$
is 
\[
	(y,z)_x=\frac{1}{2}\big(d(x,y)+d(x,z)-d(y,z)\big).
\]
We say that $(X,d)$ is $\delta$-hyperbolic if for $x,y,z $ and $u$ in $X$ we have
$$(y,z)_{x} \geq \min \lbrace (y,u)_{x},(u,z)_{x}\rbrace -\delta.$$

A sequence $(a_{n})_{n\in  \mathbb{N}}$ in $X$ converges at infinity if $(a_{i},a_{j})_{x}\rightarrow +\infty$ as $i,j$ goes to $+\infty$. We say that two sequences $(a_{n})_{n\in  \mathbb{N}}$ and $(b_{n})_{n\in  \mathbb{N}}$ are equivalent if $(a_{i},b_{j})_{x} \rightarrow +\infty$ as $i,j$ goes to $\infty$. An equivalence class of $(a_{n})_{n\in  \mathbb{N}}$ is denoted by $\lim a_{n}$ and we denote by $\partial X$ the set of equivalence classes. These definitions are independent of the choice of the basepoint $x$.

We extend the definition of the Gromov product to $\overline{X}$ by 
\begin{equation}\label{gromovextended}
(v,w)_{x}:= \sup \lim_{i,j}(a_{i},b_{j})_{x} 
\end{equation}
where the $\sup$ is taken over all sequences $(a_{n})_{n\in \mathbb{N}},(b_{n})_{n\in \mathbb{N}}$
 such that $v=\lim_{i}a_{i}$ and $w=\lim_{j}b_{j}$.
 
To sum up the property of the Gromov product in general $\delta$-hyperbolic spaces we have: 
 
 \begin{prop}\label{propGromov} \cite[3.17 Remarks, p. 433]{BH}.\\
 Let $X$ be a $\delta$-hyperbolic space and fix $x$ in $ X$.
 \begin{enumerate}

\item The extended Gromov product $(\cdot , \cdot)_{x}$ is continuous on $X \times X$ , but not necessarily on $ \overline{X} \times \overline{X}.$
\item In the definition of $(a,b)_{x}$, if we have $a$ in $X$ (or $b$ in $X$), then we may always take the respective sequence to be the constant value $a_i = a$ (or $b_{j} =b$).
\item For all $v, w$ in $\overline{X}$ there exist sequences $(a_n)$ and $(b_n)$ such that $v=\lim a_{n}$ and $w=\lim b_{n}$ and $(v,w)_x=\lim_{n \rightarrow +\infty} (a_n,b_n)_x$.

\item For all $u,v$ and $w$ in $\overline{X}$ by taking limits we still have $$(v,w)_x \geq \min{\{(v,u)_{x},(u,w)_{x} \} }-2\delta .$$
\item For all $v,w$ in $\partial X$ and all sequences $(a_i)$ and $(b_j)$ in X with $v= \lim a_ i$ and $w = \lim b_j$, we have:

$$(v,w)_{x} -2\delta \leq  \liminf_{i,j}(a_i ,b_j)_{x} \leq (v,w)_{x}.$$ 

\end{enumerate}
\end{prop}
We refer to \cite[8.- Remarque, Chapitre 7, p. 122]{G} for a proof of the statement (5) which is also an exercise, but it will be very useful in this note.

Indeed the Gromov boundary of a proper geodesic space has also a geometrical definition (\cite[Chapter III.H, p. 431]{BH}) and if $v\in \partial X$ we can pick a geodesic ray $c$, namely a map $r:\mathbb{R}_{+} \rightarrow X$ which is an isometry, such that $r(+\infty)=v$ to define the \emph{Busemann function associated to $r$} as
 \begin{equation}\label{buseman}
b_{r}(x)=\lim_{t \rightarrow +\infty} d(x,r(t))-t,
\end{equation}
which is well defined due to the triangle inequality.

We define the \emph{horoshperical distance relative to $r$} as:
\begin{equation}\label{buseman'}
B_{r}(x,y)=b_{r}(x)-b_{r}(y) .
\end{equation}
where $c$ is a geodesic ray such that $v=r(+\infty)$. 
Following the vocabulary of \cite[Definition 1.21]{CM} we define the \emph{Busemann function} as
\begin{equation}\label{buseman}
\beta_{v}(x,y)=2(v,y)_{x}-d(x,y).
\end{equation}
If follows from Proposition \ref{propGromov} (5) that for all $r$ such that $r(+\infty)=v$
\begin{align*}
\beta_{v}(x,y)-4\delta \leq B_{r}(x,y) \leq \beta_{v}(x,y).
\end{align*}
We have the following properties concerning the Busemann functions:

\begin{prop}\label{busemannprop} We have for $x,y,z \in X$ and $v\in \partial X$: 
\begin{enumerate}
\item  $0 \leq \beta_{v}(x,y)+ \beta_{v}(y,x) \leq 4\delta.$
\item  $-d(x,y) \leq \beta_{v}(x,y) \leq d(x,y)+2\delta.$
\item  $\beta_{v}(x,z) \leq \beta_{v}(x,y)+\beta_{v}(y,z) \leq \beta_{v}(x,z) +4\delta.$
\end{enumerate}
\end{prop}

Recall the notion of visual metric on the boundary: fix a basepoint $x$ and let  $\epsilon\leq \frac{\log 2}{4\delta}$. Then there exists a metric $d_{x}^{\epsilon}$ on the boundary such that for all $v,w\in \partial X$
\begin{equation}\label{metricvisu}
\frac{1}{c_{m}} {\rm e}^{ -\epsilon (v,w)_{x}}  \leq d_{x}^{\epsilon}(v,w) \leq c_{m} {\rm e}^ {-\epsilon (v,w)_{x}}
\end{equation}
for some positive constant $c_{m}>0$, see \cite[Chapitre 7, \S3]{G} for more details. 

A particular class of $\delta$-hyperbolic spaces is the class of CAT(-1) spaces. A CAT(-1) space is a proper metric geodesic space such that every geodesic triangle is thinner than its comparison triangle in the hyperbolic plane. For example, every complete simply connected Riemannian manifolds of curvature $\kappa<0$ is a CAT(-1) space.
 In the context of CAT(-1) spaces, the formula 
\begin{equation}\label{distance}
	d_x(v,w)={\rm e}^{-(v,w)_x}
\end{equation}
defines a distance on $\partial  X$ (we set $d_x(v,v)=0$). This is due to M. Bourdon, we refer to \cite[ 2.5.1 Th\'eor\`eme]{Bou} for more details.

 Once a visual parameter $\epsilon$ has been chosen  and a basepoint $x$ in $X$ has been fixed, a ball of radius $r>0$ centered at $v\in \partial X$ with respect to $d^{\epsilon}_{x}$ is denoted by $B(v,r)$.  A ball of radius $r$ centered at $x\in X$ is denoted by $B_{X}(x,r)$.

 As we said in the introduction, if $X$ is \emph{upper gromov bounded by above}, we can choose $w_{x}^{y}\in \partial X$ such that $(v,w_{x}^{y})_{x}\geq d(x,y)-c.$ In particular CAT(-1) spaces are upper Gromov bounded by above and also the hyperbolic groups (viewed as metric spaces) endowed with a left invariant word metric associated with some finite symmetric set of generators (see \cite[Lemma 4.1]{LG})  
\begin{lemma}\label{theta}
Assume that $X$ is \emph{upper Gromov bounded by above}. There exists $c>0$ such that for all $\gamma \in \Gamma$, we have 
\begin{enumerate}
\item
$ (v,y)_{x}\leq (v,w_{x}^{y})+c+2\delta .$
\item $ \min\lbrace (v,w_{x}^{y}),d(x,y) \rbrace -c-2\delta \leq (v,y)_{x}.$ 
\end{enumerate}
\end{lemma}
\begin{proof}
We prove (1). We have by Proposition \ref{propGromov} $(4)$
\begin{align*}
 (v,w_{x}^{y})_{x}&\geq \min \lbrace (v,y)_{x}, (y,w_{x}^{y})_{x}\rbrace -2\delta\\
& \geq (v,y)_{x}-c -2\delta.
\end{align*}
\end{proof}

The radial limit set $\Lambda^{rad}_{\Gamma}$ is a subset of $\Lambda_{\Gamma}$ such that a point $v\in \Lambda^{rad}_{\Gamma}$ if and only if there exists a positive constant $C>0$ and a sequence $(\gamma_{n})_{n\in \mathbb{N}}$ with $(\gamma_{n} y)_{n\in \mathbb{N}}$ converging to $v$ such that
\begin{equation}\label{radial'}
d(\gamma_{n}y,x)-(\gamma_{n}y,v)_{x}\leq C,
\end{equation}
 Notice that the definition does not depend on $x$ or $y$. We have:

\begin{lemma}\label{radial}
Assume that $X$ is upper Gromov product bounded by above. If $v\in \Lambda^{rad}_{\Gamma}$ there exist $C>0$ and a sequence $(\gamma_{n})$ in $\Gamma$ such that $(\gamma_{n}x) \in \Omega_{C}(v)$.
\end{lemma}

\begin{proof}
Let $x$ be in $GH(\Lambda_{\Gamma})$ and let $v$ be in  $\Lambda^{rad}_{\Gamma}$. There exist $C'>0$ and a sequence $(\gamma_{n})\in \Gamma$ such that
\begin{align}\label{eqexo}
d(\gamma_{n}x,x)-(\gamma_{n} x,v)_{x}\leq C'.
\end{align}

We can assume that $d(x,\gamma_{n}x)\geq 1$ for all $n$.
For all $n$ consider the point $w_{x}^{\gamma _{n} x}$ in $\partial X$ and apply Lemma \ref{theta} (1) to obtain $$ (\gamma_{n}x,v)_{x} \leq (w_{x}^{\gamma_{n} x},v)_{x}+c+2\delta .$$

 Combining the above inequality with Inequality (\ref{eqexo}) and with Inequality (\ref{metricvisu}) we obtain:
\begin{align*}
d^{\epsilon}_{x} (v,w_{x}^{\gamma_{n} x})&\leq c_{m}  {\rm e}^{\epsilon(C'+c+2\delta)} {\rm e}^{-\epsilon d(x,\gamma_{n}x)}\\
& \leq c_{m}  {\rm e}^{\epsilon(C'+c+2\delta)} d(x,\gamma_{n}x)^{\epsilon / \alpha} {\rm e}^{-\epsilon d(x,\gamma_{n}x)},
 \end{align*}
 and by Definition \ref{nonrad}, $(\gamma_{n}x) \in \Omega_{C}(v)$ with the constant $C=c_{m}  {\rm e}^{\epsilon(C'+c+2\delta)}$. This completes the proof.
\end{proof}
 It is well known that for quasi-convex cocompact groups we have $\Lambda_{\Gamma}=\Lambda^{rad}_{\Gamma}$, see for example \cite[Chapitre 1, Pr\'eliminaires 1F, Proposition 1.10]{Ro}. Notice also in CAT(-1) spaces that $GH(\Lambda_{\gamma})=CH(\Lambda_{\Gamma})$.

As we have mentioned in the introduction, if the weak inequality of Harish-Chandra holds for some $x$ then it holds for all $x$ in $X$. Indeed observe that for all $x$ and $y$, $$1+|\gamma |_{x}\leq (1+2d(x,y))(1+| \gamma |_{y})  $$ and  observe also that Proposition \ref{busemannprop} implies  $$ \forall \gamma \in \Gamma, \phi_{x}(\gamma)\leq  C\phi_{y}(\gamma),$$
for some positive constant $C>0$ which does not depend on $\gamma$.

\subsection{Standard facts about measure theory}
\subsubsection{Ahlfors regularity and Harish-Chandra function}

Let $(Z,d,m)$ be a compact metric measure space with a metric $d$ and a measure $m$. We denote by $\mbox{Diam}(Z)$ the diameter of $Z$. We say that the metric measure space $Z$ is \emph{Ahlfors $D$-regular }if there exists a positive constant $k>0$ such that for all $z$ in $Z$ and $0<r<\mbox{Diam}(Z)$ we have $$k^{-1}r^{D} \leq m(B(z,r))\leq k r^{D}.$$ If only the right hand side inequality holds, then we say that $(Z,d,m)$ is \emph{upper Ahlfors $D$-regular}.

Let $(X,d)$ be a $\delta$-hyperbolic space, let $\Gamma$ be a quasi-convex cocompact discrete group of isometries of $X$ and let $\mu=(\mu_{x})_{x \in X}$ be a $\Gamma$-invariant quasi-conformal density of dimension  $\alpha(\Gamma)$. Fix an origin $x$ in $X$ and consider the metric measure space $(\partial \Lambda_{\Gamma},d_{x}^{\epsilon},\mu_{x})$, where $d_{x}^{\epsilon}$ is a visual metric with visual parameter $\epsilon>0$. A fundamental property of the conformal densities (which goes back to D. Sullivan in the case of hyperbolic spaces $\mathbb{H}^{n}$, see \cite{Su}) and which is due to M. Coornaert in general $\delta$-hyperbolic spaces (see \cite[Proposition 7.4]{Coo}) is that the metric measure space $(\Lambda_{\Gamma},d^{\epsilon}_{x},\mu_{x})$ satisfies Ahlfors regularity conditions. More precisely, in the case of general proper and geodesic $\delta$-hyperbolic spaces, the metric measure space $(\Lambda_{\Gamma},d^{\epsilon}_{x},\mu_{x})$ is Ahlfors $D(\Gamma)$-regular where $D(\Gamma)=\alpha(\Gamma) /\epsilon$.
Using the Ahlfors $D(\Gamma)$-regularity of $(\Lambda_{\Gamma},d^{\epsilon}_{x},\mu_{x})$, we can obtain the Harish-Chandra estimates as follows:

\begin{prop}\label{H-CHestimates}
Assume that $X$is upper Gromov bounded by above.
Let $\mu=(\mu_{x})_{x\in X}$ be a $\Gamma$-invariant conformal density of dimension $\alpha(\Gamma)$, where $\Gamma$ is a quasi-convex cocompact discrete group of isometries of $X$.  Then for each $x$ in $X$ there exists $R>0$ such that the function $\varphi_{x}$ satisfies the Harish-Chandra estimates on $\Gamma x \backslash B_{X}(x,R)$.

Moreover, for each $x$ in $X$ there exists $R>0$ such that $\varphi_{x}$ satisfies the Harish-Chandra estimates on GH$(\Lambda_{\Gamma})\backslash{B_{X}(x,R)}$.
\end{prop}

The arguments coming from \cite[Proposition 3.3]{BM} combined with the assumption of   \emph{upper Gromov bounded by above} give a proof of the Harish-Chandra estimates, as we can see in \cite[Propositon 4.4]{B2} in CAT(-1) spaces or in \cite[Lemma 5.1]{LG} in the case of hyperbolic groups. 

The Harish-Chandra estimates are on GH$(\Lambda_{\Gamma})\backslash{B_{X}(x,R)}$ and not on GH$(\Lambda_{\Gamma})$ because we want to obtain polynomials with positive coefficients.

\subsubsection{Maximal inequality of type (1,1)}
Let $\nu$ be a complex measure and we denote by $\|\nu \|$ its norm. 
Define the maximal function associated with $\nu$ with respect to $\mu_{x}$ as 
\begin{equation}\label{defimaxmesure}
\mathcal{M}\nu (v):=\sup_{r>0}{ \frac{ |\nu |\big(B(v,r)\big)}{\mu_{x}\big(B(v,r)\big)} },
\end{equation}
and notice that if ${\rm d}\nu= f {\rm d}\mu_{x}$ with $f\in L^{1}(\partial X,\mu_{x})$ we denote $\mathcal{M}\nu$ by $\mathcal{M}f$ and we have: 

\begin{equation}\label{defimaxfonction}
\mathcal{M}f(v)=\sup_{r>0} \frac{1}{\mu_{x}\big(B(v,r)\big)} \int_{B(v,r)}|f(w)|{\rm d}\mu_{x}(w).
\end{equation}

\subsubsection{Covering, Vitali's lemma and maximal inequality of type $(1,1)$}
We follow \cite[Chapitre 7]{Ru} and state the following properties.
\begin{lemma} 
Let $(Z,d)$ be a metric space. Let $A\subset Z $ such that $A\subset \cup_{i \in I} B(z_{i},r_{i})$ where $I$ is a finite set. Then there exists $I^{*}\subset I$ such that for all $i\neq j$ in $I^{*}$ we have $B(z_{i},r_{i})\cap B(z_{j},r_{j})=\varnothing$ and $A\subset  \cup_{i \in I^{*}}B(z_{i},3r_{i})$.
\end{lemma}
And here is the fundamental tool of our main result.
\begin{prop}\label{maxim}\emph{Maximal inequality of type $(1,1)$.}\\
Let $(Z,d,m)$ be a metric measure space such that $m$ is a (upper) Ahlfors $D$-regular measure. Let $\nu$ be a complex measure on $Z$.
Then we have $$ m( \{ \mathcal{M} \nu >t  \} ) \leq C 3^{D}\frac{\|\nu\|}{t}.$$
\end{prop}

We refer to \cite[Th\'eor\`eme 7.4]{Ru} for a proof.

\section{ Theorem \`a la Fatou}\label{Fatouu}

\subsection{Continuity}
We recall a result obtained in \cite[Section 7]{B2} in the context of CAT(-1) spaces and give the proofs of the results in the slightly more general context of quasi-conformal \emph{continuous} densities.\\
Let $X$ be proper geodesic $\delta$-hyperbolic space which is upper Gromov product bounded by above and assume that the Gromov boundary $\partial X$ is endowed with visual metrics $(d_{x}^{\epsilon})_{x\in X}$ of visual parameter $\epsilon>0$.
Let $\Gamma$ be a discrete group of isometries of $X$ and consider $\mu$ a $\Gamma$-invariant quasi-conformal density of dimension $\alpha$. Observe that the closure of $GH(\Lambda_{\Gamma})$ in $\overline{X}$ satisifes $\overline{GH(\Lambda_{\Gamma})}\cap \partial X=\Lambda_{\Gamma}$. 
For each $x\in X$ and for all $f\in L^{1}(\partial X, \mu_{x})$ define 

\begin{equation}\label{extension}
\overline{\mathcal{P}_{0} }f: y \in \overline{GH(\Lambda_{\Gamma})} \mapsto   \left\{
    \begin{array}{ll}
      \big( \mathcal{P}_{0} f \big)(y) & \mbox{if } y \in GH(\Lambda_{\Gamma}) \\
        f(y) & \mbox{if }  y\in \Lambda_{\Gamma}. 
    \end{array}
		\right. 
		\end{equation}
		
The goal of this subsection is to prove the following proposition:		
\begin{prop}\label{continuity} 
Let $\mu$ be a $\Gamma$-invariant \emph{continuous} quasi-conformal density of dimension $\alpha$ and fix an origin $x$ in $X$. Assume that the function $\varphi_{x}$ satisfies the left hand side inequality of the Harish-Chandra estimates. Then for all continuous functions $f$ in $C(\partial X)$ the function $\overline{\mathcal{P}_{0} }f$ is continuous on $\overline{GH(\Lambda_{\Gamma})}$.
\end{prop}
We give a proof of this standard result in the case of the unit disc, in the context of $\delta$-hyperbolic spaces.
We follow a standard method based on a sequence of functions called ``mollifiers" or called  `` approximation of the identity", or also called a Dirac-Weierstra$\ss$ family.

\subsubsection{Dirac-Weierstra$\ss$ family}\label{DW}
 We follow \cite[Chapter 2, \S 2.1, p\ 46]{Jo} and adapt the definition of a Dirac-Weierstra$\ss$ family to any finite and positive measure $\nu$.
 
\begin{defi}
\label{diracweier} A Dirac-Weierstra$\ss$ family $(K(y,\cdot))_{y\in GH(\Lambda_{\Gamma})}$ with respect to $\nu$, is a continuous map
 $K: (y,v) \in GH(\Lambda_{\Gamma})\times\Lambda_{\Gamma} \mapsto K(y,v) \in \mathbb{R}$  satisfying
\begin{enumerate}
	\item $K(y,v)\geq 0$ for all $v \in \Lambda_{\Gamma}$ and $y\in GH(\Lambda_{\Gamma})$,
	\item $\int_{\Lambda_{\Gamma}} K(y,v) d\nu(v) =1$ for all $y \in GH(\Lambda_{\Gamma})$,
	\item for all $ v_{0} \in \Lambda_{\Gamma}$ and for all $ r >0$ we have: $$\int_{\Lambda_{\Gamma} \backslash B(v_{0},r)}K(y,v)d\nu(v) \rightarrow 0 ~~\mbox{as}~~ y \rightarrow v_{0}. $$
\end{enumerate}
\end{defi} 

A Dirac-Weierstra$\ss$ family yields an integral operator $\mathcal{K}$:
\begin{align*}
 \mathcal{K}:   f \in L^{1}(\Lambda_{\Gamma},\nu) & \mapsto \mathcal{K}f \in C(GH(\Lambda_{\Gamma}))
\end{align*}
defined as :
\begin{align*}
\mathcal{K}f :  y \in  GH(\Lambda_{\Gamma}) &\mapsto  \int_{\Lambda_{\Gamma}}f(v)K(y,v)d\nu(v) \in \mathbb{C} .
\end{align*}

\subsubsection{Continuity}\label{sectioncontinuity}
Let  $f$ be in $L^{1}(\Lambda_{\Gamma},\nu)$. We define the function $\overline{\mathcal{K}}f$ on $\overline{GH(\Lambda_{\Gamma})}$ as the following: 
\begin{equation}\label{extension}
\overline{\mathcal{K} }f: y \in  \overline{GH(\Lambda_{\Gamma})} \mapsto  \overline{\mathcal{K}}f(y)= \left\{
    \begin{array}{ll}
       \mathcal{K}f(y) & \mbox{if } y \in GH(\Lambda_{\Gamma}) \\
        f(y) & \mbox{if }  y \in \Lambda_{\Gamma} 
    \end{array}
		\right. 
		\end{equation}

Thus, $\overline{\mathcal{K}}$ is an operator which assigns a function defined on $\overline{GH(\Lambda_{\Gamma})}$ to a function defined on $\Lambda_{\Gamma}$.

\begin{prop}\label{continuity'}
If $f$ is a continuous functions on $\Lambda_{\Gamma}$, the function $\overline{\mathcal{K}}  (f)$ is a continuous function on $\overline{GH(\Lambda_{\Gamma})}$.
 
\end{prop}
\begin{proof}
 Since $K$ is a continuous function,  $\overline{\mathcal{K}}f$ is continuous on $GH(\Lambda_{\Gamma})$. Let $v_{0}$ be in $\Lambda_{\Gamma}$. 

Let $\varepsilon>0$ (we use $\varepsilon$ for this small quantity because $\epsilon $ is reserved for the visual parameter).  
Since $f$ is continuous, there exists $r>0$ such that $$|f(v_{0})-f(v)|<\frac{\varepsilon}{2},$$ whenever $v\in B(v_{0},r)$. Besides, by (3) in Definition \ref{diracweier}, there exists a neighborhood $V$ of $v_{0}$ such that for all $y\in V$ we have: $$\int_{\Lambda_{\Gamma} \backslash B(v_{0},r) } K(y,v) d\nu(v)\leq \frac{\varepsilon}{4\|f\|_{\infty}}\cdot$$ We have for all $y\in V$ : 
\begin{align*}
|\overline{\mathcal{K}}f(v_{0})-&\overline{\mathcal{K}}f(y)|=|f(v_{0})-\mathcal{K}(y)|=\bigg |\int_{\Lambda_{\Gamma}}f(v_{0})-f(v) K(y,v)d\nu(v) \bigg| \\
&\leq\int_{ B(v_{0},r) }|f(v_{0})-f(v)| K(y,v) d\nu(v)+\int_{\Lambda_{\Gamma} \backslash B(v_{0},r) }|f(v_{0})-f(v)| K(y,v) d\nu(v) \\ 
&\leq \varepsilon.
\end{align*}

Hence, $\overline{\mathcal{K}}f$ is a continuous function on $\overline{GH(\Lambda_{\Gamma})}$. 
\end{proof}

\subsubsection{An example of Dirac-Weierstra$\ss$ family}
We shall use the assumption Gromov product upper bounded by above to obtain the following lemma.

\begin{lemma}\label{visual} Let $v$ be in $\Lambda_{\Gamma}$. Then $d^{\epsilon}_{x}(v,w_{x}^{y})\rightarrow 0$ as $y\rightarrow v$.
\end{lemma}

\begin{proof}
 Let $y_{n}$ be a sequence of points of $X$ such that $y_{n}\rightarrow v$. 
 Lemma \ref{theta} (1) implies 
\begin{align*}
 (v,w^{y_{n}}_{x})_{x} \geq (v,y_{n})_{x}-c-\delta.
 \end{align*}
 
  Since $y_{n}\rightarrow v$, the quantity $(v,y_{n})_{x}$ goes to infinity and thus $d^{\epsilon}_{x}(v,w_{x}^{y})\rightarrow 0$ as $y\rightarrow v$.
\end{proof}

\begin{prop}\label{dirac}
Fix a base point $x$ in $X$ and consider the finite positive measure $\mu_{x}$. Assume that there exists a polynomial $Q_{1}$ (at least of degree 1) with positive coefficients, such that for all $y\in GH(\Lambda_{\Gamma}) \backslash B_{X}(x,R)$  for some $R>0$ we have $$Q_{1}\big(d(x,y)\big)\exp{\bigg(-\frac{\alpha}{2}d(x,y)\bigg)}\leq \varphi_{x}(y).$$ Then  $$\big(K(y,\cdot)\big)_{y\in GH(\Lambda_{\Gamma})}:=\bigg( \frac{P(x,y,.)^{1/2}}{ \varphi_{x}(y)} \bigg)_{y\in GH(\Lambda_{\Gamma})} $$ is a Dirac-Weierstra$\ss$ family with respect to $\mu_{x}$.
\end{prop}
\begin{proof}
First of all, notice that $K$ is a continuous function by hypothesis on the quasi-conformal density. 
The statements $(1)$ and $(2)$ in Definition \ref{diracweier} are obvious, we shall only prove $(3)$.

Let $B(v_{0},r)$ be the ball of radius $r$ at $v_{0}$ in $\partial X$.
 Let $\varepsilon>0$ (we use $\varepsilon$ for this small quantity because $\epsilon $ is reserved for the visual parameter). Since $Q_{1}$ is a polynomial at least of degree one, then there exists $R'>0$ (chosen with $R'\geq R$) such that for all $y$ satisfying $d(x,y)>R'$ we have: $$\frac{ C_{r,\alpha,\delta}\|\mu_{x} \|}{Q_{1}\big(d(x,y)\big)}<\varepsilon$$ where  $C_{r,\alpha,\delta}=\big(2^{\alpha}{\rm e}^{\alpha (2\delta+c)} C_{q} c_{m}^{\alpha / \epsilon}\big)/r^{\alpha}$ is a positive constant and where $C_{q}$ and $c_{m}$ come from Definition \ref{quasiconform} and Inequality (\ref{metricvisu}).

 Lemma \ref{visual} yields a neighborhood $V$ of $v_{0}$ such that $d_{x}(v_{0},w_{x}^{y})\leq \frac{r}{2}$ for all $y\in V$. We have  for all $v$ in $\partial X \backslash B(v_{0},r)$:
\begin{align*}
d_{x}^{\epsilon}(v,w_{x}^{y})&\geq d_{x}^{\epsilon}(v,v_{0})- d_{x}^{\epsilon}(v_{0},w_{x}^{y})\\
&\geq r-d_{x}^{\epsilon}(v_{0},w_{x}^{y})\\
&\geq \frac{r}{2}.
\end{align*}
We set $V_{R'}=V\cap X\backslash B_{X}(x,R')$.
Using the assumption \emph{upper Gromov bounded by above} with the above inequality we obtain for all $y\in V_{R'}$:

\begin{align*}
\int_{\partial X \backslash B(v_{0},r)}\frac{P(x,y,v) ^{\frac{\alpha}{2}}}{ \varphi_{x}(y)}d\mu_{x}(v) 
&\leq 
C_{q}\int_{\partial X \backslash B(v_{0},r)} \frac{{\rm e}^{ \alpha (v,y)_{x}-\frac{\alpha}{2} d(x,y) }}{ \varphi_{x}(y)} d\mu_{x}(v)\\
\mbox{ Lemma \ref{theta} (1) } &\leq C_{q} c_{m}^{\alpha /\epsilon}
\int_{\partial X \backslash B(v_{0},r)}\frac{{\rm e}^{\alpha (2\delta+c)}{\rm e}^{-\frac{\alpha}{2} d(x,y)}}{d_{x}^{\epsilon}(v,w_{x}^{y})^{\alpha} \varphi_{x}(y) } d\mu_{x}(v)\\
&\leq C_{r,\alpha,\delta}  \int_{\partial X \backslash B(v_{0},r)}\frac{{\rm e}^{-\frac{\alpha}{2} d(x,y)}}{Q_{1}\big(d(x,y)\big){\rm e}^{-\frac{\alpha}{2} d(x,y)}}d\mu_{x}(v)\\
&= C_{r,\alpha,\delta}\int_{\partial X \backslash B(v_{0},r)}\frac{1}{Q_{1}\big(d(x,y)\big)}d\mu_{x}(v)\\
&\leq\frac{ C_{r,\alpha,\delta} \| \mu_{x}\|}{Q_{1}\big(d(x,y)\big)}\\
&\leq \varepsilon.
\end{align*}

Hence $$\int_{\partial X \backslash B(v_{0},r)}\frac{P(x,y,v)^{\frac{\alpha}{2}}}{ \varphi_{x}(y)}d\mu_{x}(v)\rightarrow 0 \mbox{ as } y\rightarrow v_{0}.$$

\end{proof}
We are ready to prove Proposition \ref{continuity}.
\begin{proof}[Proof of Proposition \ref{continuity}]
Since $\Gamma$ is convex cocompact, the Harish-Chandra estimates of $\varphi_{x}$ hold on $GH(\Lambda_{\Gamma})\backslash B_{X}(x,R)$ by Proposition \ref{H-CHestimates}. Hence, Proposition \ref{continuity'} combined with Proposition \ref{dirac} complete the proof.
\end{proof} 

\subsection{Inequality on nontangential maximal function}
Let $R$ be a positive real number and define for all $C>0$ the weak nontangential maximal function associated with $f\in L^{1}(\partial X,\mu_{x})$ as
\begin{equation}\label{maxfunction}
\mathcal{N}^{R}_{C}f :v\in \partial X \mapsto \sup { \big \lbrace \big| \mathcal{P}_{0}f(y) \big| \mbox{ such that } y\in \Omega _{C}(v) \backslash B_{X}(x,R)   } \big \rbrace \in \mathbb{R}^{+}.
\end{equation}

Here is the fundamental proposition of this paper and the key ingredient is the lower bound of the Harish-Chandra function or in other words the left hand side of Harish-Chandra estimates.

\begin{prop}\label{mainprop}
Assume that the visual parameter satisfies $\epsilon \leq 1$ and 
assume that $\varphi_{x}$ satisfies the Harish-Chandra estimates on $GH(\Lambda_{\Gamma}) \backslash R$. For all $C>0$, there exists $R_{0}>0$ and a constant $C_{0}>0$ such that we have $$\mathcal{N}^{R_{0}}_{C}f \leq C_{0}  \mathcal{M}f . $$ 
\end{prop}

We follow the proof of \cite{Sjo2}.

\begin{proof}
We will write $|y|$ instead of $d(x,y)$ and define for $v_{0}\in \partial X$ and $y \in X$ the ball $B(v_{0},\rho(y))$ of $\partial X$ centered at $v_{0}$ with the radius $$\rho(y):=2C |y|^{\frac{\epsilon}{\alpha}}{\rm e}^{-\epsilon |y|}.$$
Denote by $B_{\rho}$ the ball $B(v_{0},\rho(y))$ and write $\partial X= B_{\rho} \cup (\partial X \backslash B_{\rho})$. 
Let $a$ be a positive real number such that $y\in GH(\Lambda_{\Gamma}) \backslash B_{X}(x,R)  $ implies $\varphi_{x}(y)\geq Q_{1}(|y|)\geq a |y|$. Let $f$ be a positive function in $L^{1}(\partial X, \mu_{x})$ and denote by $C_{q}$ a constant coming from Definition \ref{quasiconform} (2) of quasi-conformal density. We have for all $ y \in GH(\Lambda_{\Gamma}) \backslash B_{X}(x,R)$ :

\begin{align*}
\frac{1}{C_{q}} \int_{B_{\rho}}f(v)\frac{P(x,y,v)^{\frac{\alpha}{2}}}{\varphi_{x}(y)} d\mu_{x}(v) & \leq \int_{ B_{\rho}}f(v) \frac{ {\rm e}^{\frac{\alpha}{2} \beta_{v}(x,y)}   }{Q_{1}(|y|){\rm e}^{-\frac{\alpha}{2}|y|}}d\mu_{x}(v) \\
\mbox{ Proposition \ref{busemannprop} (2) } & \leq {\rm e}^{\alpha \delta}   \int_{ B_{\rho} }f(v) \frac{ {\rm e}^{\frac{\alpha}{2} |y|}  }{Q_{1}(|y|){\rm e}^{-\frac{\alpha}{2}|y|}}d\mu_{x}(v) \\
& \leq   \frac{  {\rm e}^{\alpha \delta} }{Q_{1}(|y|){\rm e}^{-\alpha|y|}}  \int_{B_{\rho}}f(v) d\mu_{x}(v)\\
& \leq   \frac{  {\rm e}^{\alpha \delta}  }{a|y|{\rm e}^{-\alpha|y|}}  \int_{B_{\rho}}f(v) d\mu_{x}(v)\\ 
\mbox{ Ahlfors $\alpha(\Gamma) / \epsilon$-regularity of $(\Lambda_{\Gamma},d^{\epsilon}_{x},\mu_{x})$ } & \leq   \frac{ (2C)^{\alpha / \epsilon} {\rm e}^{\alpha \delta}  }{a \times k} \times \bigg( \frac{1}{\mu_{x}(B_{\rho})}  \int_{B_{\rho}}f(v) d\mu_{x}(v) \bigg)\\
&\leq \bigg( \frac{ (2C)^{\alpha / \epsilon }  {\rm e}^{\alpha \delta} }{a\times k}  \bigg) \mathcal{M}f (v_{0}),
\end{align*}

Hence for all $y \in GH(\Lambda_{\Gamma}) \backslash B_{X}(x,R)$ : 
\begin{equation}\label{ineg1fatou}
 \frac{1}{C_{q}}\int_{B_{\rho}}f(v)\frac{P(x,y,v)^{\frac{\alpha}{2}}}{\varphi_{x}(y)} d\mu_{x}(v)  \leq \bigg( \frac{ (2C)^{\alpha / \epsilon }  {\rm e}^{\alpha \delta} }{a k}  \bigg) \mathcal{M}f (v_{0}).
\end{equation}

If $ y\in \Omega_{C}(v)$ we have $d^{\epsilon}_{x}(w_{x}^{y},v_{0}) \leq C |y|^{\frac{\epsilon}{\alpha}}\exp{(- \epsilon|y|)}=\frac{1}{2}\rho(y)$.
Thus, for $v$ satisfying $d^{\epsilon}_{x}(v,v_{0})\geq \rho(y)$ we have $$ d_{x}^{\epsilon}(v,v_{0}) \leq  d_{x}^{\epsilon}(v,w^{y} _ {x} )+\frac{1}{2} d_{x}^{\epsilon}(v,v_{0}),$$
and hence 

\begin{equation}\label{inegtriang}
 d^{\epsilon}_{x}(v,w^{y}_{x})\geq \frac{d^{\epsilon}_{x}(v,v_{0})}{2} \cdot 
\end{equation}

Then, for the second term of the integral on $\partial X \backslash B_{\rho}$ we have for $y \in GH(\Lambda_{\Gamma}) \backslash B_{X}(x,R)$
\begin{align*}
\frac{1}{C_{q}}\int_{\partial X \backslash B_\rho}f(v)\frac{P(x,y,v)^{\frac{\alpha}{2}}}{\varphi_{x}(y)} d\mu_{x}(v)& \leq \int_{d^{\epsilon}_{x}(v,v_{0})>\rho(y)}f(v) \frac{ {\rm e}^{\alpha (v,y)_{x}-\frac{\alpha}{2} |y|}  }{Q_{1}(|y|){\rm e}^{-\frac{\alpha}{2}|y|}}d\mu_{x}(v) \\
&= \frac{1}{Q_{1}(|y|)}\int_{d^{\epsilon}_{x}(v,v_{0})>\rho(y)}f(v)  {\rm e}^{\alpha (v,y) _{x}} d\mu_{x}(v) \\
 \mbox{ Lemma \ref{theta} (1) } &\leq \frac{{\rm e}^{\alpha(2\delta+c)}}{a|y|}\int_{d^{\epsilon}_{x}(v,v_{0})>\rho(y)}f(v)  {\rm e}^{\alpha (v,w_{x}^{y})_{x} } d\mu_{x}(v) \\
   & \leq  \frac{{\rm e}^{\alpha(2\delta+c)} c_{m}^{\alpha / \epsilon }}{a|y|}\int_{ d^{\epsilon}_{x}(v,v_{0})>\rho(y)}f(v) \frac{ 1 }{{d^{\epsilon}_{x} (v,w_{x}^{y})^{\alpha}}} d\mu_{x}(v)\\
\mbox{ Inequality (\ref{inegtriang}) } & \leq    \frac{ 2^{\alpha}  {\rm e}^{\alpha(2\delta+c)} c_{m}^{ \alpha / \epsilon }} {a} \times \frac{1}{|y|} \int_{d_{x}^{\epsilon}(v,v_{0})>\rho(y)}f(v) \frac{ 1 }{d^{\epsilon}_{x} (v,v_{0})^{\alpha}} d\mu_{x}(v).
\end{align*}

 We work now with the term $$ \int_{d_{x}^{\epsilon}(v,v_{0})>\rho(y)}f(v) \frac{ 1 }{d^{\epsilon}_{x} (v,v_{0})^{\alpha}} d\mu_{x}(v).$$
Define $N(y)$ the integer such that $N(y)={\rm{E} } \big(\frac{ \log(D / \rho(y))}{\log 2}\big)+1$, where $D$ denotes the diameter of $\partial X \backslash B_{\rho} $ and ${\rm{E} }$ denotes the integer part. Take $R^{'}>0$ such that $|y|\geq R'$ implies $N(y)\leq c' |y|$ for some positive constant $c'>0$. We may assume that $\|\mu_{x}\|=1$ (if not, do the same computation with $\nu_{x}=\frac{\mu_{x}}{\|\mu_{x}\|}$).
We have for $\epsilon \leq 1$:
\begin{align*}
\int_{d_{x}^{\epsilon}(v,v_{0})>\rho(y)}f(v) \frac{ 1 }{d^{\epsilon} _{x}(v,v_{0})^{\alpha}} d\mu_{x}(v)&=\sum_{k=0}^{N(y)}  \int_{2^{k+1}\rho(y)>d_{x}^{\epsilon}(v,v_{0})>2^{k}\rho(y)}f(v) \frac{ 1 }{d^{\epsilon}_{x} (v,v_{0})^{\alpha}} d\mu_{x}(v)\\
&\leq\frac{1}{ \rho(y)^{\alpha}} \sum_{k=0}^{N(y)} \frac{1}{2^{\alpha k}} \int_{2^{k+1}\rho(y)>d_{x}^{\epsilon}(v,v_{0})>2^{k}\rho(y)}f(v) d\mu_{x}(v)\\
&=2^{\alpha} \sum_{k=0}^{N(y)} \frac{1}{ (2^{ k+1}\rho(y))^{\alpha}} \int_{2^{k+1}\rho(y)>d_{x}^{\epsilon}(v,v_{0})>2^{k}\rho(y)}f(v) d\mu_{x}(v)\\
&\leq2^{\alpha}c_{m}^{\epsilon} \sum_{k=0}^{N(y)}  \frac{1}{\mu_{x}\big(B(v_{0},2^{k+1}\rho(y))\big)^{\epsilon}}\int_{2^{k+1}\rho(y)>d_{x}^{\epsilon}(v,v_{0})}f(v) d\mu_{x}(v)\\
&\leq2^{\alpha}c_{m}^{\epsilon} \sum_{k=0}^{N(y)}  \frac{1}{\mu_{x}\big(B(v_{0},2^{k+1}\rho(y))\big)}\int_{2^{k+1}\rho(y)>d_{x}^{\epsilon}(v,v_{0})}f(v) d\mu_{x}(v)\\
&\leq 2^{\alpha}c_{m}^{\epsilon}  \times N(y)  \times \mathcal{M}f (v_{0}).
\end{align*}
Set $R_{0}:= \max \{R, R' \}$. It follows from the definition of $N(y)$  that $y \in GH(\Lambda_{\Gamma}) \backslash B_{X}(x,R_{0})$ implies
\begin{equation}\label{ineq''}
\frac{1}{C_{q}}\int_{\partial X \backslash B_{\rho}}f(v)\frac{P(x,y,v)^{\frac{\alpha}{2}}}{\varphi_{x}(y)} d\mu_{x}(v) \leq \bigg( \frac{ 2^{2\alpha}  {\rm e}^{\alpha(2\delta+c)} c_{m}^{ \alpha  } c'} {a}   \bigg)  \mathcal{M}f(v_{0}) .
\end{equation}
Inequality (\ref{ineg1fatou}) combined with Inequality (\ref{ineq''}) give a constant $C_{0}$ to complete the proof.

\end{proof}
\begin{proof}[Proof of Theorem \ref{Fatou}]
We follow the proof of \cite[Th\'eor\`eme 7.7, p. 172]{Ax}. Fix $x$ in $X$ and let $f$ in the Banach space $L^{1}(\partial X, \mu_{x})$ endowed with its standard norm $\|\cdot\|_{1}$. Define for all $y$ in $X$:
\begin{equation}
\mathcal{L}_{y}f:v \in \Lambda_{\Gamma} \mapsto  \big| \big(\overline{ \mathcal{P}_{0}}f\big) (y)-f(v) \big| \in \mathbb{R}^{+}
\end{equation}
and for $C>0$, define
\begin{equation}
\mathcal{L}_{C}f :v\in  \Lambda_{\Gamma} \mapsto \limsup_{\substack{y\rightarrow v \\ y\in \Omega_{C}(v)}} \big(\mathcal{L}_{y}f\big)(v)\in \mathbb{R}^{+}.
\end{equation}
We shall prove that $\mathcal{L}_{C}f=0$ almost everywhere with respect to $\mu_{x}$.
Let $\varepsilon>0$ and pick a continuous function $g$ in $C(\partial X)$ such that $\|f-g \|_{1}\leq \varepsilon$. Observe that
  Proposition \ref{continuity} implies that $\mathcal{L}_{C}g=0$ and observe also that for all $C>0$ and for $R>0$: $$\mathcal{L}_{C}f \leq \mathcal{N}_{C}^{R}f +|f|.$$  
 We have
 \begin{align*}
 \mathcal{L}_{C}f&=\mathcal{L}_{C}(f-g)+\mathcal{L}_{C}g\\
 &=\mathcal{L}_{C}(f-g)\\
&\leq \mathcal{N}_{C}^{R_{0}}(f-g)\\
& \leq C_{0} \mathcal{M}(f-g)+|f-g|,
\end{align*}
where the two positive constants $R_{0}$ and $C_{0}$ are yielded by Proposition \ref{mainprop} and the last inequality follows from Proposition \ref{mainprop}.
The maximal inequality of Proposition \ref{maxim} implies that for all $s>0$, we have 
\begin{align*}
\mu_{x} \big(\lbrace \mathcal{L}_{C}f >s\rbrace \big)& \leq 3^{\alpha}kC_{0}\frac{\| f-g\|_{1}}{s}+\frac{\| f-g\|_{1}}{s}.  
\end{align*}
Since $\|f-g \|_{1}\leq \varepsilon$, the above inequality implies that for all $s>0$ we have $$ \mu_{x} \big( \lbrace  \mathcal{L}_{C} f >s\rbrace \big) =0.$$ Hence, $$\mathcal{L}_{C}f=0$$ almost everywhere for all $C>0$. To conclude,  
define $A_{n}:=\{ \mathcal{L}_{n} f=0 \}$ for $n\in \mathbb{N^{*}}$. We have shown that $A_{n}$ is a set of full measure with respect to $\mu_{x}$ and thus $\cap_{n\geq 1 } A_{n}$ is also a set of full measure. Hence the proof is complete.
\end{proof}

\section{Application to decay of matrix coefficients of boundary representations}\label{applic}
\subsection{The weak inequality of Harish-Chandra}
Let $\Gamma$ be a convex cocompact group of isometries of a proper CAT(-1) space $X$ and let $\mu$ be a $\Gamma$-invariant conformal density of dimension $\alpha(\Gamma)$.\\
Consider a matrix coefficient associated with the boundary representation given by two essentially bounded functions $f,g \in L^{\infty}(\partial X)$, namely  $  \langle \pi_{x}(\cdot)f,g \rangle$ and observe that $|\langle \pi_{x}(\gamma)f,g\rangle|\leq \| f\|_{\infty} \|g\|_{\infty}\phi_{x}(\gamma)$ for all $\gamma \in \Gamma$. In other words, the matrix coefficient $  \langle \pi_{x}(\cdot)f,g \rangle$ satisfies the weak inequality of Harish-Chandra of degree $0$.
Nevertheless, for all functions $\xi$ which are not in $L^{\infty}$ the matrix coefficient $\langle \pi_{x}(\cdot)\textbf{1}_{\partial X},\xi \rangle $ can not satisfy the weak inequality of degree $0$ anymore. The proof is a direct application of theorem \`a la Fatou:
\begin{proof}[Proof of Corollary \ref{Fatou1}]
Let $x$ be in $GH(\Lambda_{\Gamma})$, let $\xi $ be in $L^{2}(\partial X,\mu_{x})$ and assume that the Harish-Chandra's weak inequality of degree $0$ holds: there exists $C>0$ such that for all $\gamma$ in $\Gamma$ we have $$\frac{\langle \pi_{x}(\gamma)\textbf{1}_{\partial X}, \xi \rangle}{\phi_{x}(\gamma)}\leq C .$$
Observe that for all $\gamma$ in $\Gamma$
 $$(\mathcal{P}_{0}\xi) (\gamma x)=\frac{\langle \pi_{x}(\gamma)\textbf{1}_{\partial X},\xi \rangle}{\phi_{x}(\gamma)} \cdot$$

For all $v $ is in $\Lambda^{rad}_{\Gamma}$, we know by  Lemma \ref{radial} that there exist  $C'>0$ and a sequence $(\gamma_{n})_{n \in \mathbb{N}} \in \Gamma$ such that $\gamma_{n}x \in  \Omega_{C'}(v)$ for all $n \in \mathbb{N}$.  Thus Theorem \ref{Fatou} implies that $\xi $ is in $L^{\infty}(\partial X)$.
Hence the weak inequality of Harish-Chandra of degree one must fail if  $\xi$ is not in $L^{\infty}(\partial X)$.
\end{proof}
\begin{remark}\label{distortion}
Let $\Gamma$ be a non-uniform lattice (in rank one) acting on its symmetric space endowed with its Riemannian left-invariant metric and consider its boundary representation associated with the Lebesgue measure class. Then it cannot satisfy the weak inequality of Harish-Chandra of any degree. Because if the weak inequality holds, then it would follow that the boundary representation satisfies property RD and the presence of one parabolic element is an obstruction (see for example \cite[Corollary 1.3]{B1}). Nevertheless, as we will see in Proposition \ref{propFatou} and \ref{proprob}, the decay of matrix coefficients is related to the property RD (see Remark \ref{RDrem}). 
\end{remark}
\subsection{Some results about the decay of matrix coefficients of boundary representations} Let $\Gamma$ be in $\mathcal{C}$ and consider a $\Gamma$-invariant quasi-conformal density of dimension $\alpha(\Gamma)$. Fix an origin $x$ in $X$ and denote $C_{n,\rho}(x)$ by $C_{n,\rho}$.

\subsubsection{Using a uniform boundedness}

Following the ideas of \cite{BM}, to prove the irreducibility of boundary representations  in \cite{B2} and in \cite{LG}, it is proved that a sequence of functions is uniformly bounded with respect to the $L^{\infty}$ norm. Recall the definition of such a sequence of functions $(F_{n,\rho})_{n> \rho}$:  
\begin{equation}
F_{n,\rho}:v\in \partial X \mapsto \frac{1}{|C_{n,\rho}|}\sum_{\gamma \in C_{n,\rho}}\frac{\pi_{x}(\gamma) \textbf{1}_{\partial X}}{\phi_{x}(\gamma)}(v) \in \mathbb{R}^{+}.
\end{equation}
where $\rho$ is a positive constant.

\begin{prop}\label{uniform}
If $\Gamma$ is in $\mathcal{C}$, then there exists $\rho>0$, $N_{\rho} \in \mathbb{N}$ and a positive constant $M>0$ such that $$\|F_{n,\rho} \|_{\infty}\leq M $$ for all $n>N_{\rho}$.
\end{prop}

Indeed, we shall use the dual inequality.

 \begin{prop}\label{poissontransform'}
If $\Gamma$ is in $\mathcal{C}$, then there exists $\rho>0$, $N_{\rho}\in \mathbb{N}$ and a positive constant $M>0$ such that for all $f\in L^{1}(\partial X, \mu_{x})$ we have  $$\frac{1}{|C_{n,\rho}|} \bigg |\sum_{\gamma \in C_{n,\rho}}\mathcal{P}_{0}(f) \bigg | \leq M \|f\|_{1} $$ for all $n>N_{\rho}$.
\end{prop}

\begin{proof}
The map $$F \in L^{\infty} \mapsto \ell_{F} \in (L^{1})^{*}$$ where $\ell_{F}(f)=\langle f,F \rangle $, is an isomorphism of Banach spaces, namely $$\sup_{f \in L^{1}} \frac{ |\ell_{F}(f)|}{\| f\|_{1}}=\|F\|_{\infty} .$$ Take the dual inequality of Proposition \ref{uniform} via the Banach space isomorphism $\ell_{F_{n,\rho}}$ to obtain the result of Proposition \ref{poissontransform'}.
\end{proof}

In Proposition \ref{uniform} and in Proposition \ref{poissontransform'} we may assume that $\rho>1$. Indeed, we may choose all $\rho\geq \rho_{0}$ where $\rho_{0}$ guarantees that $|C_{n,\rho_{0}}|>0$ for $n$ large enough. \\

We need the following lemma:

\begin{lemma}\label{cspoisson} 
 Let $\mu=(\mu_x)_{x\in X}$ be a $\Gamma$-equivariant quasi-conformal density of dimension $\alpha$, where $\Gamma$ is a discrete group of isometries of $X$. Let $x$ be in $X$ and let $\xi,\eta$ be in $L^2(\partial X,\mu_x)$ be two positive functions and let $\gamma$ be in $\Gamma$. Then
\[
 \frac{\left\langle \pi_x(\gamma)\xi,\eta\right\rangle^2}{\phi_{x}(\gamma)^2}\leq
(\mathcal{P}_0\xi^2)(\gamma^{-1} x)(\mathcal{P}_0\eta^2)(\gamma  x).   
\]
\end{lemma}
\begin{proof} The Cauchy-Schwarz inequality implies :
	\begin{align*}
	 \left\langle \pi_x(\gamma)\xi,\eta\right\rangle^2
	&=\left(\int_{\partial X}\xi(\gamma^{-1}v)
	P(x,\gamma  x,v)^{\frac{\alpha}{2}}\eta(v)d\mu_x(v)\right)^2\\
	&\leq \bigg(\int_{\partial X}\xi(\gamma^{-1}v)^2
	P(x,\gamma x,v)^{\frac{\alpha}{2}}d\mu_x(v)\bigg)\bigg(
	\int_{\partial X}
	P(x,\gamma  x,v)^{\frac{\alpha}{2}}\eta(v)^2d\mu_x(v)\bigg)\\
	&=\left\langle \pi_x(\gamma)\xi^2,1_{\partial_X}\right\rangle
	\left\langle \pi_x(\gamma)1_{\partial_X},\eta^2\right\rangle\\
	&=\left\langle \pi_x(\gamma^{-1})1_{\partial_X},\xi^2\right\rangle
	\left\langle \pi_x(\gamma)1_{\partial_X},\eta^2\right\rangle.
	\end{align*}
Divide this inequality by
 $\phi_{x}(\gamma)^2$ and use Definition (\ref{sjogrennot}) to complete the proof.   
	
\end{proof}

We can prove Proposition \ref{propFatou}:

\begin{proof}[Proof of Proposition \ref{propFatou}]

It is sufficent to prove that there exists a polynomial $Q$ of degree $2$ such that  for all positive $\xi $ and for all sufficiently large integers $k$ 
\begin{equation}\label{inegintermed}
\sum_{C_{k,\rho}}\langle \pi_{x}(\gamma)\textbf{1}_{\partial X},\xi \rangle^{2}\leq Q(k)\|\xi\|_{2}^{2}.
\end{equation}
Recall that there exists $R>0$ such that the Harish-Chandra estimates hold on $GH(\Lambda_{\Gamma}) \backslash B_{X}(x,R)$.
Thus all integers $k$ such that $k \geq R +\rho$ we have

\begin{align*}
\sum_{C_{k,\rho}}\langle \pi_{x}(\gamma)\textbf{1}_{\partial X},\xi \rangle^{2}&=\sum_{C_{k,\rho}}\phi_{x}^{2}(\gamma) \frac{  \langle\pi_{x}(\gamma)\textbf{1}_{\partial X},\xi \rangle^{2} }{\phi_{x}^{2}(\gamma)}\\
\mbox{(upper bound of the Harish-Chandra estimates)}&\leq Q^{2}_{2}(k)  {\rm e}^{-\alpha(\Gamma) k} \sum_{C_{k,\rho}}  \frac{  \langle\pi_{x}(\gamma)\textbf{1}_{\partial X},\xi \rangle^{2} }{\phi_{x}^{2}(\gamma)}\\
\mbox{(Lemma \ref{cspoisson})}&\leq Q^{2}_{2}(k) {\rm e}^{-\alpha(\Gamma) k} \sum_{C_{k,\rho}}  \frac{  \langle\pi_{x}(\gamma)\textbf{1}_{\partial X},\xi^{2} \rangle }{\phi_{x}(\gamma)}\\
&= Q^{2}_{2}(k) {\rm e}^{-\alpha(\Gamma) k} \sum_{C_{k,\rho}} \mathcal{P}_{0}(\xi^{2})\\
&\leq C' Q^{2}_{2}(k) \frac{1}{|C_{k,\rho}|} \sum_{C_{k,\rho}} \mathcal{P}_{0}(\xi^{2})\\
\end{align*}

The last inequality follows from the fact that for convex cocompact groups, hyperbolic groups and lattices in rank one, we can find $C'>0$ such that $${\rm e}^{-\alpha(\Gamma) k} \leq C' \frac{1}{|C_{k,\rho}|}\cdot $$ Apply Proposition \ref{poissontransform'} to the last inequality of the above computation to obtain Inequality (\ref{inegintermed}) with a polynomial $Q$ of degree $2$.
This finishes the proof. 
\end{proof}

\subsubsection{With a naive application of a theorem of T. Roblin}\label{roblin}
Recall the equidistribution theorem that we shall use and that we have already used in the paper \cite{B2}. We refer to \cite[Preliminaires]{B2} for the definition of arithmetic spectrum and for the definition of Bowen-Margulis-Sullivan measure. Notice that convex cocompact groups with non-arithmetic spectrum and lattices in rank one semisimple Lie groups satisfy the hypothesis of the following theorem.

Let $D_{z}$ be the unit Dirac mass at $z$ and let $m_{\Gamma}$ be the Bowen-Margulis-Sullivan measure.

 \begin{theorem}\label{roblin}(T. Roblin)
Let $\Gamma$ be a discrete group of isometries of $X$ with a non-arithmetic spectrum. Assume that $\Gamma$ admits a finite Bowen-Margulis-Sullivan measure associated with a $\Gamma$-invariant conformal density $\mu$ of dimension $\alpha=\alpha(\Gamma)$. Then, for all $x,y \in X$ we have: $$\alpha e^{-\alpha n}||m_{\Gamma}||\sum_{ \left\{\gamma \in \Gamma|d( x,\gamma  y)< n  \right\}}D_{\gamma^{-1}  x} \otimes D_{\gamma  y} \rightharpoonup \mu_{x} \otimes \mu_{y} $$ as $n\rightarrow +\infty$  with respect to the weak* convergence of $C(\overline{X} \times \overline{X})^{*}$.
\end{theorem}

\begin{proof}[Proof of Proposition \ref{proprob}]

Assume first that $x$ is in $GH(\Lambda_{\Gamma})$, let $\rho>0$ and let $f$ and $g$ in $L^{2}(\partial X,\mu_{x})$. Recall that $C_{n,\rho}(x)=\{  n-\rho \leq |\gamma|_{x}<n+\rho \}$.
Lemma \ref{cspoisson} implies $$\left|\left\langle \pi_{x}(\gamma)f,g \right\rangle\right|^{2}\leq \varphi_{x}(\gamma)^{2} (D_{\gamma^{-1} x} \otimes D_{\gamma  x} )(\mathcal{P}_{0}f^{2}\otimes \mathcal{P}_{0}g^{2}). $$
If $\gamma$ is in $C_{n,\rho}(x)$, the upper bound of the Harish-Chandra estimates (see Proposition \ref{H-CHestimates}) implies that there exists $Q_{2}$ such that
$$\varphi_{x}(\gamma)^{2}\leq Q_{2}(n)e^{-\alpha(\Gamma) (n+\rho)}.$$

It follows that for n large enough we have: $$\sum_{\gamma \in C_{n,\rho}(x)}\left|\left\langle \pi_{x}(\gamma)f,g \right\rangle\right|^{2}\leq Q_{2}(n) {\rm e}^{-\alpha(\Gamma) (n+\rho)}\sum_{\gamma \in C_{n,\rho}(x)}  (D_{\gamma^{-1} x} \otimes D_{\gamma x} )(\mathcal{P}_{0}f^{2}\otimes \mathcal{P}_{0}g^{2}) .$$
If $f$ and $g$ are continuous functions, the lower bound of Harish-Chandra estimates  combined with Proposition \ref{continuity} imply that $\overline{\mathcal{P}_{0}}f^{2}\otimes \overline{\mathcal{P}_{0}}g^{2}$
 is a continuous function on $\overline{X}\times \overline{X}$. Set $$Q=\frac{1}{\alpha(\Gamma)\|m_{\Gamma} \|} Q_{2}\cdot $$ Since $C_{n,\rho}(x)\subset \Gamma_{n+\rho}(x)$, we have for $n$ large enough: 
$$ \frac{1}{Q(n)}\sum_{\gamma \in C_{n,\rho}(x)}\left|\left\langle \pi_{x}(\gamma)f,g \right\rangle\right|^{2}\leq \|m_{\Gamma}\|\alpha(\Gamma) {\rm e}^{-\alpha(\Gamma)(n+\rho)}\sum_{\gamma \in \Gamma_{n+\rho}(x)}  (D_{\gamma^{-1} x} \otimes D_{\gamma x} )(\overline{\mathcal{P}_{0}}f^{2}\otimes \overline{\mathcal{P}_{0}}g^{2}) .$$ 
 Theorem \ref{roblin} implies for all $\rho>0$:
 \begin{align*}
\limsup_{n\rightarrow +\infty} \frac{1}{Q(n)}\sum_{\gamma \in C_{n,\rho}(x)}\left|\left\langle \pi_{x}(\gamma)f,g \right\rangle\right|^{2}&\leq \lim_{n\rightarrow +\infty} \|m_{\Gamma}\|\alpha(\Gamma) {\rm e}^{-\alpha(\Gamma) (n+\rho)}\sum_{\gamma \in \Gamma_{n+\rho}(x)}  (D_{\gamma^{-1} x} \otimes D_{\gamma x} )(\overline{\mathcal{P}_{0}}f^{2}\otimes \overline{\mathcal{P}_{0}}g^{2})\\
 &=(\mu_{x}\otimes \mu_{x})(\overline{\mathcal{P}_{0}}f^{2}\otimes \overline{\mathcal{P}_{0}}g^{2})\\
 &=\mu_{x}(\overline{\mathcal{P}_{0}}f^{2}) \mu_{x}(\overline{\mathcal{P}_{0}}g^{2}) \\
 &=\mu_{x}(f^{2}) \mu_{x}(g^{2}) \\
 &=\| f \|^{2}_{2} \|g\|^{2}_{2}. 
 \end{align*}
 Hence we have proved for all $x$ in $GH(\Lambda_{\Gamma})$, forall $\rho>0$, there exists $ Q $ such that  
  \begin{equation}\label{stp1}
 \limsup_{n\rightarrow +\infty} \frac{1}{Q(n)}\sum_{\gamma \in C_{n,\rho}(x)}\left|\left\langle \pi_{x}(\gamma)f,g \right\rangle\right|^{2} \leq \| f \|^{2}_{2} \|g\|^{2}_{2}.
 \end{equation}
 If $x$ is not in $GH(\Lambda_{\Gamma})$, pick $x'$ in $GH(\Lambda_{\Gamma})$ and observe that 
 \begin{equation}\label{couronne}
 C_{n,\rho}(x)\subset C_{n,\rho+2d(x,x')}(x').
 \end{equation}
 
 Then, define the multiplication operator $M_{x,x'}$ defined as 
 \begin{equation}\label{operatormul}
 M_{x,x'}:\xi \in L^{2}(\partial X, \mu_{x}) \mapsto m_{x,x'} \xi \in L^{2}(\partial X, \mu_{x'}),
 \end{equation}
 where $$m_{x,x'}:v\in \partial X \mapsto {\rm e}^{\frac{\alpha}{2} \beta_{v} (x,x')},$$
 which is an isometry and which intertwines $\pi_{x}$ and $\pi_{x'}$, namely $$\pi_{x'}M_{x,x'}=M_{x,x'}\pi_{x}.$$ Fix $\rho>0$ and use the operator $M_{x,x'}$ to apply Inequality (\ref{stp1}) with $x'$, $\rho+2d(x,x')$ to complete the proof. 
 
 \end{proof}

\section{Harish-Chandra's Schwartz algebra of discrete groups}\label{HCS}
\subsection{Results}
Recall that the space $C_{c}(\Gamma)$ is an algebra for the convolution product defined as $$f_{1}\ast f_{2} (g)=\sum_{\gamma \in \Gamma} f_{1}(\gamma)f_{2}(\gamma^{-1}g),~~\forall g\in \Gamma.$$ It is natural to ask if the Harish-Chandra's Schwartz space $\mathcal{S}_{t}(\Gamma)$, for $t$ large enough, has a structure of a convolution algebra. The answer is positive. The only difficulty was to find suitable conditions which play analogous roles of the conditions $(HC3c)$ and $(HC3d)$ in \cite[Chapitre 4, p. 82]{L}. We replace the averaging process over the compact $K$ in the context of semisimple Lie groups by sums over annuli, based on Proposition \ref{poissontransform'}. \\

Fix $x$ an origin in $X$ and denote $d(x,\gamma x)$ by $|\gamma|$, $\{\gamma \in  \Gamma  \mbox{ such that } |\gamma|\leq n \}$ by $\Gamma_{n}$ and $\Gamma_{n} \backslash \Gamma _{n-1}$ by $C_{n}$.

\subsection{An inequality for the Harish-Chandra function}
\begin{lemma}\label{trick2}
 Let $\Gamma$ be in $\mathcal{C}$. There exists $t_{0}$ such that for all $t\geq t_{0}$ there exists $C_{t}$ satisfying for all $g$ in $\Gamma$: \begin{align*}\sum_{ \gamma \in \Gamma } \frac{\phi_{x}(g\gamma^{-1})\phi_{x}(\gamma)}{(1+|\gamma|)^{t}}\leq C_{t} \phi_{x}(g).
\end{align*}
\end{lemma}

\begin{proof}[Proof.]
 Let $g$ be in $\Gamma$ and observe that for all $t>0$ and for all integer $N$ we have:
\begin{equation}\label{sommedecoup}
\sum_{ \gamma \in \Gamma } \frac{\phi_{x}(g\gamma^{-1})\phi_{x}(\gamma)}{(1+|\gamma|)^{t}} =\sum_{ \gamma \in \Gamma_{N} } \frac{\phi_{x}(g\gamma^{-1})\phi_{x}(\gamma)}{(1+|\gamma|)^{t}} +\sum^{\infty}_{n=N+1} \sum_{ \gamma \in C_{n} } \frac{\phi_{x}(g\gamma^{-1})\phi_{x}(\gamma)}{(1+|\gamma|)^{t}}. 
\end{equation}
Consider the first term: observe that Cauchy-Schwarz inequality implies that for all $\gamma \in \Gamma$ we have $\phi_{x}(\gamma)\leq 1$. Then observe also that $\phi_{x}(g\gamma^{-1})=\langle \pi_{x}(g) (\pi(\gamma^{-1})\textbf{1}_{\partial X}), \textbf{1}_{\partial X} \rangle$. Since the function $\pi(\gamma^{-1})\textbf{1}_{\partial X}$ is a positive continuous function on a compact set and since the representation $\pi_{x}$ preserves the cone of positive function we have $$ \phi_{x}(g\gamma^{-1})\leq C_{\gamma}\phi_{x}(g), $$ where $C_{\gamma}:=\sup \lbrace \pi(\gamma^{-1})\textbf{1}_{\partial X}(v), v\in \partial X \rbrace$ is some positive constant. Hence for all $t>0$ and for all $N$ there exists a positive constant $C'$ such that \begin{equation}\label{term1}
 \sum_{ \gamma \in \Gamma_{N} } \frac{\phi_{x}(g\gamma^{-1})\phi_{x}(\gamma)}{(1+|\gamma|)^{t}}\leq C'\phi_{x}(g).
\end{equation}

Consider now the second term:\\
Since $\Gamma$ is in $\mathcal{C}$ recall  that for all $\rho>0$ there exists a polynomial $Q$ such that for all integers $n$ large enough and for all $\gamma \in C_{n,\rho}$ we have: 
\begin{equation}\label{HCHvolume}
\phi_{x}(\gamma)\leq Q(n)\frac{1}{|C_{n,\rho}|}.
\end{equation}
 
Let $\rho>1 $ such that $C_{n}\subset C_{n,\rho}$.
Using Inequality (\ref{HCHvolume})  we have for all $n\geq \rho $:
\begin{align*}
\sum_{ \gamma \in C_{n} } \phi_{x}(g\gamma^{-1})\phi_{x}(\gamma)&=\sum_{ \gamma \in C_{n} }\phi_{x}(\gamma)^{2} \frac{\phi_{x}(g\gamma^{-1})}{\phi_{x}(\gamma)}\\
&\leq \frac{Q(n)}{|C_{n,\rho}|}\sum_{ \gamma \in C_{n} }  \frac{\phi_{x}(g\gamma^{-1})}{\phi_{x}(\gamma)}\\
&=  \frac{Q(n)}{|C_{n,\rho}|}\sum_{ \gamma \in C_{n,\rho} } \frac{\left\langle \pi_{x}(g\gamma^{-1})1,1\right\rangle}{\phi_{x}(\gamma)}\\
&= \frac{Q(n)}{|C_{n,\rho}|}\sum_{ \gamma \in C_{n,\rho} } \frac{\left\langle \pi_{x}(\gamma)1,\pi_{x}(g^{-1})1\right\rangle}{\phi_{x}(\gamma)}\\
&= \frac{Q(n)}{|C_{n,\rho}|}\sum_{ \gamma \in C_{n,\rho} } D_{\gamma x}( \mathcal{P}_{0}f)
\end{align*}
where $$f=\pi_{x}(g^{-1})1\in L^{1}(\partial X,\mu_{x}).$$ \\
 Observe that Equality (\ref{inverse}) implies $$||\pi_{x}(g^{-1})1||_{1}=\phi_{x}(g^{-1})=\phi_{x}(g).$$
 Applying Propostion \ref{poissontransform'} to $f$ we have for $n$ large enough: $$ \sum_{ \gamma \in C_{n} }\phi_{x}(g\gamma^{-1})\phi_{x}(\gamma)\leq Q(n)\phi_{x}(g).$$
 
Thus there exists $t_{0}$ such that for all $t\geq t_{0}$ there exists a constant $C_{t}>0$ satisfying for all $g\in \Gamma$: $$\sum_{ \gamma \in \Gamma } \frac{\phi_{x}(g\gamma^{-1})\phi_{x}(\gamma)}{(1+|\gamma|)^{t}}\leq C_{t} \phi_{x}(g).$$
\end{proof}
\subsection{Convolution algebra}
\begin{prop} \label{Laff1}
Let $\Gamma$ be in $\mathcal{C}$. There exists $t_{0}>0$, such that for all $t\geq t_{0}$ we have  $\mathcal{S}_{t}(\Gamma)\ast \mathcal{S}_{t}(\Gamma)\subset \mathcal{S}_{t}(\Gamma)$.
\end{prop}

\begin{proof}[Proof.] 
We follow the proof of V. Lafforgue in \cite[Proposition 4.3]{L} . Let $t>0$ and let $f_{1},f_{2}$ be in $\mathcal{S}_{t}(\Gamma)$. We have for all $\gamma \in \Gamma$ $$|f_{i}(\gamma)|\leq C_{i}\frac{\phi_{x}(\gamma)}{(1+|\gamma|)^{t}}$$ for some $C_{i}>0$ with $i=1,2$. Thus for all $g\in \Gamma$ we have
\begin{align*}
\frac{1}{C_{1}C_{2}}|f_{1}&\ast f_{2} (g)|=\frac{1}{C_{1}C_{2}}\sum_{\gamma \in \Gamma} f_{1}(\gamma)f_{2}(\gamma^{-1}g)\\
&\leq  \sum_{\gamma \in \Gamma} \frac{\phi_{x}(\gamma)}{(1+|\gamma|)^{t}}\frac{\phi_{x}(\gamma^{-1}g)}{(1+|\gamma^{-1}g|)^{t}}\\
&= \sum_{ \left\{|\gamma|\leq \frac{|g|}{2} \right\}} \frac{\phi_{x}(\gamma)}{(1+|\gamma|)^{t}}\frac{\phi_{x}(\gamma^{-1}g)}{(1+|\gamma^{-1}g|)^{t}} +\sum_{ \left\{ |\gamma|> \frac{|g|}{2} \right\}} \frac{\phi_{x}(\gamma)}{(1+|\gamma|)^{t}}\frac{\phi_{x}(\gamma^{-1}g)}{(1+|\gamma^{-1}g|)^{t}}.
\end{align*}
Observe that $ |\gamma|\leq \frac{|g|}{2} \Rightarrow |\gamma^{-1}g|\geq \frac{|g|}{2}$.		
							
We have for the first sum:
\begin{align*}
\sum_{ \left\{ |\gamma|\leq \frac{|g|}{2} \right\}} \frac{\phi_{x}(\gamma)}{(1+|\gamma|)^{t}}\frac{\phi_{x}(\gamma^{-1}g)}{(1+|\gamma^{-1}g|)^{t}} &\leq\frac{2^{t}}{(1+|g|)^{t}} \sum_{ \left\{ |\gamma|\leq \frac{|g|}{2} \right\}} \frac{\phi_{x}(\gamma)\phi_{x}(\gamma^{-1}g)}{(1+|\gamma|)^{t}} \\
& \leq\frac{ 2^{t}}{(1+|g|)^{t}} \sum_{ \gamma \in \Gamma} \frac{\phi_{x}(\gamma) \phi_{x}(\gamma^{-1}g)}{(1+|\gamma|)^{t}}.
\end{align*}			
									
We have for the second sum:
\begin{align*}
\sum_{ \left\{ |\gamma|> \frac{|g|}{2} \right\}} \frac{\phi_{x}(\gamma)}{(1+|\gamma|)^{t}}\frac{\phi_{x}(\gamma^{-1}g)}{(1+|\gamma^{-1}g|)^{t}}
&\leq \frac{2^{t}}{(1+|g|)^{t}}  \sum_{ \left\{ |\gamma|> \frac{|g|}{2} \right\}} \frac{\phi_{x}(\gamma)\phi_{x}(\gamma^{-1}g)}{(1+|\gamma^{-1}g|)^{t}}\\
& \leq \frac{2^{t}}{(1+|g|)^{t}}  \sum_{ \gamma \in \Gamma } \frac{\phi_{x}(\gamma)\phi_{x}(\gamma^{-1}g)}{(1+|\gamma^{-1}g|)^{t}}\\
&= \frac{2^{t}}{(1+|g|)^{t}}\sum_{ \gamma \in \Gamma } \frac{\phi_{x}(g\gamma^{-1})\phi_{x}(\gamma)}{(1+|\gamma|)^{t}}
\end{align*}							
where the last inequality comes from the change of variable $\gamma'=\gamma^{-1}g$.	\\

			Lemma \ref{trick2} implies that the two terms are bounded by $\phi_{x}(g)$ for all $t\geq t_{0}$ for some positive real number $t_{0}$. Hence, there exists $t_{0}$ such that for all $t\geq t_{0}$ we have for all $g\in \Gamma$: $$f_{1}\ast f_{2} (g)\leq  B_{t}\frac{\phi_{x}(g)}{(1+|g|)^{t}},$$ for some positive constant $B_{t}$.						
								\end{proof}									
	\begin{prop}\label{Laff2}
There exists $t_{0}>0$, such that for all $t\geq t_{0}$ we have  $\mathcal{S}_{t}(\Gamma)\ast \ell^{2}( \Gamma)\subset \ell^{2}( \Gamma)$.
\end{prop}								
\begin{proof}[Proof.]
We follow the proof of V. Lafforgue in \cite[Proposition 5.2]{L}.
Let $h$ be in $\ell^{2}(\Gamma)$. It is enough to prove that for all $t$ large enough there exists a constant $C_{t}>0$ such that $$\sum_{g\in \Gamma} \left( \sum_{\gamma \in \Gamma} \frac{\phi_{x}(\gamma)}{(1+|\gamma|)^{t}}h(\gamma ^{-1} g)\right)^{2}\leq C_{t}\|h\|^{2}_{2}. $$

First, we use Cauchy-Schwarz inequality to obtain: 
\begin{equation}\label{CSine}
\left( \sum_{\gamma \in \Gamma} \frac{\phi_{x}(\gamma)}{(1+|\gamma|)^{t}}h(\gamma ^{-1} g)\right)^{2}\leq \left( \sum_{\gamma \in \Gamma} \frac{\phi_{x}(\gamma)}{(1+|\gamma|)^{t}}\frac{h^{2}(\gamma ^{-1} g)}{\phi_{x}(\gamma^{-1}g)}\right)\left( \sum_{\gamma \in \Gamma} \frac{\phi_{x}(\gamma^{-1}g) \phi_{x}(\gamma)}{(1+|\gamma|)^{t}} \right).
\end{equation} 

Lemma \ref{trick2} implies that there exists $t_{0}>0$ such that for all $t\geq t_{0}$ there exists $C_{t}$ satisfying for all $g\in \Gamma$:
\begin{equation}\label{inegC}
 \sum_{\gamma \in \Gamma} \frac{\phi_{x}(\gamma^{-1}g) \phi_{x}(\gamma)}{(1+|\gamma|)^{t}} \leq C_{t} \phi_{x}(g). 
\end{equation}

Combining Inequality (\ref{CSine}) with Inequality (\ref{inegC}) we obtain for all $t\geq t_{0}$:
\begin{align*}
\sum_{g\in \Gamma} \left( \sum_{\gamma \in \Gamma} \frac{\phi_{x}(\gamma)}{(1+|\gamma|)^{t}}h(\gamma ^{-1} g)\right)^{2}&\leq C_{t} \sum_{(\gamma,g) \in \Gamma \times \Gamma} \frac{\phi_{x}(\gamma)}{(1+|\gamma|)^{t}}\frac{h^{2}(\gamma ^{-1} g)}{\phi_{x}(\gamma^{-1}g)}\phi_{x}(g) \\
&= C_{t} \sum_{(\gamma,\gamma') \in \Gamma \times \Gamma} \frac{\phi_{x}(\gamma)}{(1+|\gamma|)^{t}}\frac{h^{2}(\gamma')}{\phi_{x}(\gamma')}\phi_{x}(\gamma \gamma') \\
&= C_{t} \sum_{\gamma' \in \Gamma} \left( \sum_{\gamma \in \Gamma} \frac{\phi_{x}(\gamma)\phi_{x}(\gamma \gamma')}{(1+|\gamma|)^{t}}\right)  \frac{h^{2}(\gamma')}{\phi_{x}(\gamma')}\\
&\leq C_{t}^{2}\|h\|^{2}_{2}< +\infty,
\end{align*}
where the last equality comes from a slight modification of Lemma \ref{trick2} using Equality (\ref{inverse}). The proof is complete.
\end{proof} 
\begin{remark}\label{c*alg}
Proposition \ref{Laff2} proves that we have the following representation  of $\mathcal{S}_{t}(\Gamma)$ for $t$ large enough: \begin{align*}\mathcal{S}_{t}(\Gamma)&\rightarrow B(\ell^{2}(\Gamma))\\
f&\mapsto M_{f}
\end{align*}
where $M_{f}\xi=f \ast \xi$ for all $\xi \in \ell^{2}(\Gamma)$ and $\|M_{f}\xi\|_{2}\leq C\|f\|_{S_{t}(\Gamma)} \|\xi\|_{2}$ for some positive constant $C$ depending on $t$. It follows that $$\|M_{f}\|_{op} \leq C \|f\|_{S_{t}(\Gamma)},$$ where $\|\cdot\|_{op}$ denotes the operator norm.
\end{remark}

The proof of Proposition \ref{LG} is almost done:

\begin{proof}[Proof.]
Proposition \ref{Laff1} proves that $\mathcal{S}_{t}(\Gamma)$ has the structure of a convolution algebra for $t$ large enough. Proposition \ref{Laff2} shows that $\mathcal{S}_{t}(\Gamma)$ for $t$ large enough is represented in the algebra of bounded operators on $\ell^{2}(\Gamma)$. Since $C_{c}(\Gamma)\subset \mathcal{S}_{t}(\Gamma)$ for all $t>0$ and thanks to Remark \ref{c*alg} it is clear that $\mathcal{S}_{t}(\Gamma)$, for $t$ large enough, generates on $\ell^{2}(\Gamma)$ the reduced $C^*$-algebra and the von Neumann algebra of $\Gamma$.
\end{proof}

\end{document}